\documentclass[journal,12pt,draftclsnofoot,onecolumn]{IEEEtran}
\usepackage{amsfonts}
\usepackage{amsmath, amsthm}
\usepackage{amssymb}
\usepackage{cite}
\usepackage{dsfont}
\usepackage{latexsym}
\usepackage{times}
\usepackage{url}
\usepackage{verbatim}

\usepackage{graphicx}
\newtheorem{theorem}{Theorem}

\newtheorem{definition}{Definition}
\newtheorem{lemma}{Lemma}

\def\bb0{{\mathbb{0}}}

\def\bb{{\mathbf{b}}}

\def\b0{{\mathbf{0}}}

\def\sf0{{\mathsf{0}}}

\def\rm0{{\mathrm{0}}}

\usepackage[]{hyperref}
\hypersetup{pdftex,colorlinks=true,allcolors=blue}
\usepackage{hypcap}
\usepackage{enumitem}
\usepackage[usenames, dvipsnames]{color}

\usepackage{graphicx,epstopdf}
\usepackage{caption,subcaption}% http://ctan.org/pkg/{caption,subcaption}

\usepackage[ruled,vlined]{algorithm2e}
\usepackage{algorithmic}
\usepackage{xr-hyper}
\usepackage{hyperref}
\usepackage{mathtools}

\usepackage{xr}
\title{Embedding a Heavy-Ball type of Momentum into the Estimating Sequences}
\author{{Endrit~Dosti, {\it Student Member, IEEE}, Sergiy~A.~Vorobyov, {\it Fellow, IEEE}, and~Themistoklis~Charalambous, {\it Senior Member, IEEE}}
	\thanks{E. Dosti and S. A. Vorobyov are with the Department of Signal Processing and Acoustics, Aalto University, Finland, e-mail:  firstname.lastname@aalto.fi. {\it (Corresponding author is Sergiy A. Vorobyov.)}}% <-this % stops a space
	\thanks{T. Charalambous is with the Department of Electrical and Computer Engineering, University of Cyprus, Cyprus, and the Department of Electrical Engineering and Automation, Aalto University, Finland, e-mail: firstname.lastname@aalto.fi.}}% <-this % stops a space

\begin{document}
	\maketitle
	
	\begin{abstract}
		We present a new accelerated gradient-based method for solving smooth unconstrained optimization problems. The goal is to embed a heavy-ball type of momentum into the Fast Gradient Method (FGM). For this purpose, we devise a generalization of the estimating sequences, which allows for encoding any form of information about the cost function that can aid in further accelerating the minimization process. In the black box framework, we propose a construction for the generalized estimating sequences, which is obtained by exploiting the history of the previously constructed estimating functions. From the viewpoint of efficiency estimates, we prove that the lower bound on the number of iterations for the proposed method is $\mathcal{O} \left(\sqrt{\frac{\kappa}{2}}\right)$. Our theoretical results are further corroborated by extensive numerical experiments on various types of optimization problems, often dealt within signal processing. Both synthetic and real-world datasets are utilized to demonstrate the efficiency of our proposed method in terms of decreasing the distance to the optimal solution, as well as in terms of decreasing the norm of the gradient. 
	\end{abstract}
	
	\IEEEpeerreviewmaketitle

	%%%%%%%%%%%%%%%%%%%%%%%%%%%%%%%%%%%%%%%
	\section{Introduction}
	\label{sec:intro}
	%%%%%%%%%%%%%%%%%%%%%%%%%%%%%%%%%%%%%%%
	
	A large number of problems arising in different scientific disciplines, such as, signal and image processing, communications, data analysis and machine learning (see \cite{6879577, 6879615, 8119874, 9053189, 9028239, gu2017projected} and references therein), can be cast as the minimization of a real-valued smooth and convex objective function:
	\begin{equation} 
		\label {opt_prob}
		\underset{x \in \mathcal{R}^n}{\text{minimize}} \: f(x), 
	\end{equation}
	where $f:\mathcal{R}^n \rightarrow \mathcal{R}$ is a $\mu$-strongly convex function with $L$-Lipschitz continuous gradient defined by a black-box oracle. As the size of these problems increases, it becomes necessary to resort to iterative methods for finding the optimal solution denoted by $x^*$. In large-scale optimization, a central role is played by first-order algorithms \cite{Beck_book}. In this work, we draw attention to the new generalized estimating sequences and convergence analysis for accelerated first-order methods in their purity. Therefore, we focus on constructing a first-order method for solving  the problem of minimizing smooth and strongly convex objectives as given in \eqref{opt_prob}.\footnote{The results obtained in this paper can be extended as well to solve composite objective problems with a non-smooth term, which is an issue that will be addressed in a later work.} Within this class of methods, one of the most important breakthroughs is the Fast (or Accelerated) Gradient Method (FGM) \cite{Nesterov_83}. Under the assumption of known parameters $\mu$ and $L$, the method reaches the complexity of $\mathcal{O}\left(\sqrt{\kappa}\right)$, where $\kappa = \frac{L}{\mu}$ is the condition number. In view of classic complexity theory for convex optimization by Nemirovski and Yudin \cite{Nemirovski_Yudin}, the method is also optimal in the sense that it minimizes the number of calls of a first-order oracle.

	Interest in FGM surged again with the paper on smoothing techniques \cite{nesterov2005smooth}. Therein, a smooth approximation of a non-smooth objective function is constructed, and then FGM is used to efficiently find the optimal solution. Following up on the work, several extensions were proposed. In \cite{Auslender_Teboulle}, the authors have proposed a class of interior gradient algorithms which exhibit an $\mathcal{O} \left(\frac{1}{k^2}\right)$ global convergence rate. In \cite{FISTA} and \cite{Nesterov_2007}, FGM has been extended to solve convex composite objectives. Another important aspect of FGM-type methods is the robustness to incorrect computation of the gradient of the objective function. It has been shown that FGM suffers from error accumulation, and to preserve the improved convergence rate only small gradient noise can be tolerated \cite{Aspremont, devolder2014first, schmidt2011convergence}. A precise characterization of the lower bounds on the performance of noisy FGMs for the case of ill-conditioned quadratic objective functions has recently been reported in \cite{9137636}.
	
	More recently, motivated by the need to construct even faster algorithms to solve large-scale problems with smooth objective functions, new perspectives of FGM and different reasons behind acceleration have been discussed, leading to new algorithms that achieve the optimal rate \cite{pmlr-v40-Flammarion15, Su_Boyd_Candes, WibisonoE7351, Allen-Zhu, Bubeck_geod, Fazel, Recht}. In \cite{pmlr-v40-Flammarion15, Su_Boyd_Candes, WibisonoE7351}, for example, the continuous-time limit of FGM is modeled as a second-order ordinary differential equation. Another perspective of FGM appears in \cite{Allen-Zhu}, where it is shown that FGM can be obtained by exploiting the linear coupling between gradient and mirror descent. In \cite{Bubeck_geod}, the authors have developed an alternative accelerated gradient method, which is inspired by the ellipsoid method. Links between the method introduced in \cite{Bubeck_geod} and FGM have been established in \cite{Fazel}. In \cite{Recht}, the convergence rates for FGM using theory from robust control have been derived. Utilizing the framework proposed in \cite{Recht}, in \cite{van2017fastest} the authors have introduced the Triple Momentum Method (TMM). The method is defined only for $\mu > 0$, and for the class of smooth and strongly convex functions it enjoys a faster convergence rate than FGM. However, as demonstrated in \cite[Table 2]{van2017fastest}, the constant terms present in the bound on the number of iterations needed until convergence for TMM depend on the condition number of the problem. In the case of ill-conditioned problems, it exceeds the bound of FGM, thus requiring more iterations to converge.

	A novel approach for analyzing the worst-case performance of first-order black-box optimization methods has appeared in \cite{drori2014performance}. The analysis conducted therein relies on the observation that the worst-case accuracy improvement of a black-box method is itself an optimization problem, and can be cast as a semidefinite program (SDP). In \cite{taylor2017smooth}, the authors utilize convex interpolation tools to show that the worst-case accuracies provided by the SDP are tight. Within this framework, for the class of smooth convex functions, i.e., $\mu = 0$, optimal methods have been presented in \cite{kim2016optimized} and \cite{drori2017exact}. Promising, yet unpublished frameworks, have been presented in \cite{taylor2021optimal} and \cite{park2021factor}, wherein the authors claim to have derived optimal methods for minimizing smooth and strongly convex optimization problems. Moreover, in \cite{Ryu}, the authors have identified several geometric structures, which are satisfied by a wide range of accelerated first-order methods. Based on these structures, they have introduced several accelerated methods, and have established their efficiency in decreasing the norm of the gradient for problems with composite objective functions. Despite the optimistic results that have appeared within these frameworks, their applicability remains limited to the design of first-order methods for convex optimization problems. On the other hand, the acceleration idea which was first introduced in FGM, has also been extended to other optimization settings, such as stochastic optimization \cite{NIPS2009_3817, kulunchakov:hal-01993531, lan2012optimal}, non-Euclidean optimization \cite{ahn2020nesterov, pmlr-v75-zhang18a}, higher-order methods \cite{baes2009estimate, fast_newton_method, nesterov2020inexact} and non-convex optimization \cite{doi:10.1137/17M1114296, ghadimi2016accelerated, li2015accelerated}. Moreover, large number of new applications have further extended the reach of the idea \cite{ji2009accelerated, boom, uribe2020dual, ye2020multiconsensus}. 
	
	An optimization method is considered optimal if it enjoys the following properties: \textit{i}) it exhibits the accelerated convergence rate; \textit{ii}) it reaches a complexity that is proportional to the lower complexity bounds. For the case of first-order methods, the complexity bounds have been introduced in \cite{Nemirovski_Yudin}. Several frameworks for constructing such methods have already been presented in the literature \cite{Nesterov_83, Nesterov_88, nesterov2005smooth}, and a unified analysis of the methods has been introduced in \cite{Tseng}. The exact worst-case performance of the method introduced in \cite{Nesterov_83} has been thoroughly characterized in \cite{taylor2017smooth}. The framework introduced in \cite{Nesterov_88} has been further studied in \cite{Nesterov_book}, wherein it is argued that the key behind constructing optimal methods is the accumulation of global information of the function that is being minimized. For this purpose, the estimating sequences are introduced. They consist of the pair $\{\phi_{k}(x)\}_{k=1}^\infty$, $\{\lambda_{k}\}_{k=1}^\infty$ and allow for parsing information around carefully selected points at each iteration, while also measuring the rate of convergence of the iterates. In the case of first-order methods, this intuition is provably correct; however, the construction of the estimating sequences is not unique, and finding a better construction, in the sense that it leads to more efficient methods, is an open question. A simple, self-contained and unified framework for the study of methods devised within the estimating sequence framework has been introduced in \cite{baes2009estimate}. Therein, the author shows how several accelerated schemes can be obtained, and provide some guidelines on the design of further estimating sequence methods. Evidently, picking the right functions to construct the estimating sequences, can lead to much faster algorithms. For example, the variant of FGM constructed in \cite{Nesterov_book} and its extension to convex composite objectives, i.e., the Accelerated Multistep Gradient Scheme (AMGS) \cite{Nesterov_2007} have been constructed using different variants of estimating sequences and are both optimal methods. The link between the two estimating sequences, as well as its implications, has been investigated in \cite{Iulian_1} and \cite{Iulian_2}.
	
	Despite being based on different variants of estimating sequences, both FGM and AMGS share the fact that the update of iterates at step $k + 1$ is done by utilizing the information available at step $k$. From the theory of the heavy-ball method \cite{Polyak}, it is known that parsing information from iterates at step $k-1$ can accelerate the minimization process. Naturally, the following question arises: ``Is it possible to explicitly embed information from earlier iterates into FGM?". We answer this question affirmatively, and propose a way to generalize the design of estimating sequences by including a newly introduced heavy-ball type of momentum term in them.\footnote{Our framework, however, can be thought as a general way of encoding any form of information about the objective function that can aid in further accelerating the minimization process.} We show that embedding our proposed type of heavy-ball momentum term into Nesterov's acceleration framework leads to a more powerful class of algorithms. Our main contributions here are the following.
	\begin{itemize}
		\item From the theoretical perspective, we show that the original construction of the estimating functions can be generalized by incorporating extra terms that depend on the previous iterates.
		\item To establish the properties of the newly introduced generalized estimating sequences, we revise the key lemmas and results established for the classical estimating sequences. Moreover, we utilize novel tools to introduce new results, as well as more intuition behind the design of estimating sequence methods.
		\item Within the black-box framework, we present a new type of heavy-ball momentum, which is captured by the newly introduced sequence of quadratic functions. Unlike the classical method introduced in \cite{Polyak}, wherein the heavy-ball momentum is utilized to stabilize the oscillations of the iterates, our proposed type of heavy-ball momentum is utilized for stabilizing the estimating sequences.
		\item From the algorithmic perspective, we develop a new method and show that (in black-box framework) it allows for embedding a heavy-ball type of momentum into FGM. Moreover, we show that FGM can be obtained as the special case when the memory terms are not considered.
		\item In terms of convergence guarantees, we show that the original results obtained for FGM can be improved. We prove that our proposed method is also an optimal method, and show that its lower bound on the number of iterations converges to $\sqrt{\frac{L}{2 \mu}} \left(\text{ln} \left(\frac{\mu R_0^2}{2 \epsilon}\right) + \text{ln}(5)\right)$, where $R_0  =||x_0 - x^*||$ and the accuracy $\epsilon \leq \frac{\mu}{2} R_0^2$. In other words, from the viewpoint of efficiency estimates, our proposed method outperforms FGM by at least a factor of $\frac{1}{\sqrt{2}}$.
		\item {Our proposed convergence analysis allows for initializing the parameter $\gamma_0 = 0$. Note that in the case of FGM, the convergence of the method was proved only when $\gamma_0 \in [\mu, 3L + \mu]$. As shown in Section \ref{numericals}, this yields an improvement over FGM. At the same time, it also makes the initialization of the proposed method more robust to the imperfect knowledge of $\mu$.}
		\item From applications perspective, we show through extensive simulations the efficiency of utilizing our method to solve various problems using both synthetic and real-world datasets.
	\end{itemize} 
	
	The paper is organized as follows. Preliminaries and reasoning needed for our developments are given in Section~\ref{Preliminaries, the strongly convex case}. In Section~\ref{SFGM}, we then develop the new method by embedding the heavy-ball momentum into the estimating sequence framework. Section~\ref{Convergence analysis} is devoted to the convergence analysis of the proposed method. Numerical study of our proposed method based on some problems frequently appearing in signal processing is performed in Section~\ref{numericals}. The paper ends with discussion and appendices.

	%%%%%%%%%%%%%%%%%%%%%%%%%%%%%%%%%%%%
	\section{Preliminaries and intuition}
	\label{Preliminaries, the strongly convex case}
	%%%%%%%%%%%%%%%%%%%%%%%%%%%%%%%%%%%%
	
	Intuitively, being optimal implies that one is making use of the available information in the best way possible. Bearing this in mind, we can start exploiting the information we have at iteration $k = 0$. In the convex setting, since $f: \mathcal{R}^n \rightarrow \mathcal{R}$, then for any non-trivial point $x_1 \in \text{dom}(f)$, that constitutes for the next iterate, three things can happen:
	\begin{enumerate}
		\item $f(x_1) > f(x_0) \geq f(x^*)$, which at first glance is not desirable as it is producing points that are away from $x^*$;
		\item $f(x_1) = f(x_0)$, which would also not be desirable as it suggests that no progress towards $x^*$ was made; 
		\item \label{cond.3} $f(x^*) \leq f(x_1) < f(x_0)$, which is desirable as the next iterate is closer to the solution of our problem $x^*$.
	\end{enumerate}

	Methods that produce sequences that always satisfy condition~\ref{cond.3} are called relaxation methods \cite{Nesterov_book}. The typical approach consists of parsing gradient information, which is the direction of steepest descent of the function, from a first-order oracle. Then, stepping in the opposite direction, where the function value must decrease, yields the next iterate. This greedy approach is widely used for solving optimization problems of the same type as \eqref{opt_prob}. The Gradient Method (GM) belongs to this family of methods, and it is easy to show that it produces a sequence of points $x_k, \; k = 1, 2, \ldots$ that converges to $x^*$ at a linear rate \cite{Boyd-Vandenberghe-04}.   
	
	However, the greedy approach of solving a convex optimization problem is not optimal. This is made more precise in \cite{Nesterov_book}, wherein it is argued that relaxation itself is too microscopic to guarantee convergence in an optimal fashion. Instead, it is suggested that optimal methods must make use of global topological properties of the objective function. This intuition is also confirmed by the performance of second-order methods. As can be seen from \cite[Fig. 9.19]{Boyd-Vandenberghe-04}, Newton's method is constructing ellipsoids around each iterate, which aid in correcting the search direction. Therein, the ellipsoids are obtained by exploiting the information contained in the Hessian of the objective function. In the case of first-order methods, such information about the Hessian is not available. Therefore, instead of constructing ellipsoids, one can consider constructing balls in the locality of the iterate, which allow for accounting for any feasible direction. This suggests utilizing an isotropic scanning function, which at step $k= 0$ would be: $\Phi_0: \mathcal{R}^n \rightarrow \mathcal{R}$. All that is known about this function is the following: 
	\begin{equation}
		\nabla^2 \Phi_0(x) = \gamma_0 I,  \label{scan_funct_cond}
	\end{equation}
	where $x \in \text{dom}(\Phi_0)$, $\gamma_0$ is the scanning radius of the ball, and $I$ is the identity matrix of size $n \times n$. Then, integrating \eqref{scan_funct_cond} twice over $x$, the following construction is obtained:
	\begin{equation}
		\Phi_0(x) = \Phi_0^* + \frac{\gamma_0}{2}||x - x_0||^2, \label{scan_funct}
	\end{equation}
	where $\Phi_0^*$ is the integration constant that characterizes the value of the function $\Phi_0 (x)$ when $x = x_0$, and $||\cdot||$ denotes the $l_2$ norm. As we will see in the sequel, recursively constructing such simple functions as \eqref{scan_funct}, which are referred to as \textit{scanning functions} in the sequel, is an integral component in the construction of the estimating sequences.
	
	Next, we can exploit the information coming from the fact that the cost function is $L$-smooth and $\mu$-strongly convex. Let $\mathcal{I} \subseteq \text{dom} (f)$ and $x, y \in \mathcal{I}$. Then, from \cite[Theorem 2.1.5]{Nesterov_book} we have
	\begin{align}
		\label{upper_bound}
		0 \leq f(x) - f(y) - \nabla f(y)^T (x - y) \leq \frac{L}{2} ||y-x||^2.
	\end{align}
	Moreover, from the definition of strongly convex function \cite[Definition 2.1.3]{Nesterov_book}, we can write
	\begin{align}
		\label{lower_bound}
		f(x) \geq f(y) + \nabla f(y)^T (x - y) + \frac{\mu}{2} ||y - x||^2.  
	\end{align}
	The above bounds suggest the need of utilizing gradient and function evaluation oracles. Throughout the paper, we assume that the computational cost of computing the gradient is comparable to the cost of computing the function values.

	%%%%%%%%%%%%%%%%%%%%%%%%%%%%%%%%%%
	\section{The Proposed Method}
	\label{SFGM}
	%%%%%%%%%%%%%%%%%%%%%%%%%%%%%%%%%%
	In this section, we first generalize the original construction of estimating sequences, and show how they can be computed recursively. Then, based on the new construction of the generalized estimating sequences, we devise our method. We conclude the section by presenting the convergence results and proof of optimality for the proposed method. 
	
	Let us begin by defining the generalized estimating sequences as follows. 
	\begin{definition}
		\label{def__1}
		The sequences $\{\Phi_{k} (x)\}_{k = 0}^\infty$ and $\{\lambda_{k}\}_{k = 0}^\infty$, $\lambda_{k} \geq 0$, are called generalized estimating sequences of the function $f(\cdot)$, if $\exists \psi_k: \mathcal{R}^n \to \mathcal{R}_{+}$, $\lambda_{k} \rightarrow 0$, and $\forall x \in \mathcal{R}^n$, $\forall k \geq 0$ we have 
		\begin{equation}
			\label{def_1}
			\Phi_{k} (x) \leq\lambda_{k} \Phi_{0}(x) + (1 - \lambda_{k}) \left(f(x) - \psi_k (x)\right).
		\end{equation}
	\end{definition} 
	Unlike the classical definition of estimating sequences utilized for constructing FGM \cite[Definition 2.2.1]{Nesterov_book}, the introduction of $\psi_k (x)$ allows for encoding any form of information about the objective function that will be useful in improving the speed at which $x_k \rightarrow x^*$. One can also think of it as a control sequence that, at each iteration, modifies the function that is to be optimized. This modification can be done in several ways, e.g., in white-box implementations $\psi_k (x)$ can be some prior information about the structure of $f(x)$, that would make the resulting function $f(x) - \psi_k(x)$ easier to optimize. In the black-box framework, which is central to our paper, such prior information is not available. Nevertheless, as we will show later, other choices are also possible. For now, we note that by setting $\psi_k (x) = 0, \forall k$, we recover the estimating sequence structure used for FGM. In this sense, Definition~\ref{def__1} is a generalization of the classical estimating sequences.

	Now, we show that the generalized estimating sequences also allow for measuring the convergence rate to optimality. 
	\begin{lemma}
		\label{SFGM_lemma_1}
		If for some sequence of points $\{x_k\}_{k = 0}^\infty$ we have $f(x_k) \leq \Phi_{k}^* \! \triangleq \! \underset{x \in {\mathcal{R}^n}}{ \text{min} } \Phi_{k} (x)$, then $f(x_k) - f(x^*) \leq \lambda_{k} \left[ \Phi_{0}(x^*) - f(x^*) \right] - \left(1 - \lambda_k\right) \psi_k(x^*)$.
	\end{lemma}
	\begin{proof}
		See Appendix \ref{Proof of Lemma 1}.  
	\end{proof}
	
	Before showing how to form the generalized estimating sequences, let us define
	\begin{align}
		\label{Psi_definition}
		\Psi_{k} \triangleq   \begin{cases}
			\underset{m \in \{1, 2, \ldots k\}, \; x \in {\mathcal{R}^n}}{ \text{sup}} \psi_{m} (x),       & \quad \text{if } k > 0, \\
			0,  & \quad \text{otherwise}.
		\end{cases}
	\end{align}
	In words, the term $\Psi_k$ is the tightest upper bound on the finite values of $\psi_k(x)$ that will be formed throughout the entire minimization process. At this point, we are ready to show how to construct the generalized estimating sequences.
	\begin{lemma}
		\label{SFGM_lemma_2}
		Assume that there exist sequences $\{\alpha_k\}_{k = 0}^\infty$, where $\alpha_{k} \in (0, 1)$ $\forall k$, $\sum_{k = 0}^{\infty} \alpha_{k} = \infty$, $\{y_k\}_{k = 0}^\infty$ and $\{\psi_k (x)\}_{k = 0}^\infty$ such that $\psi_k (x) \geq 0$ $\forall k = 0, 1, \ldots$. Let $\psi_0(x) = 0$ and $\lambda_{0}$ = 1. Then, the sequences $\{\Phi_{k} (x)\}_{k = 0}^\infty$ and $\{\lambda_{k}\}_{k = 0}^\infty$, which are defined recursively as
		\begin{align}
			\label{lambda_recursive}
			\lambda_{k+1} &= (1 - \alpha_k) \lambda_{k}, \\ 
			\Phi_{k+ 1} (x) &= (1 - \alpha_k) \left( \Phi_{k} (x) + \psi_k (x) \right) - \psi_{k + 1}  (x) - \Psi_k + \alpha_k \psi_{k} (x)\nonumber \\  +  &\alpha_{k} \left( f(y_k) + \nabla f(y_k)^T (x - y_k)  + \frac{\mu}{2} ||x - y_k||^2\right) ,\label{phi_k+1_SFGM}
		\end{align}
		are generalized estimating sequences.
	\end{lemma}
	\begin{proof}
		See Appendix \ref{Proof of Lemma 2}.
	\end{proof}
	Different from the earlier results summarized in \cite{Nesterov_book}, Lemma~\ref{SFGM_lemma_1} has the following benefits. First, since $\lambda_{k} \geq 0, \; \forall k$, it clarifies why the construction of the regularizing term should be such that $\psi_k (x) \geq 0, \; \forall k = 0, 1, \ldots$. Second, it shows that the convergence rate to optimality now depends on both the sequence $\{\lambda_k\}_{k=0}^{\infty}$ and the sequence $\{\psi_k (x)\}_{k=0}^{\infty}$. Furthermore, the result of Lemma \ref{SFGM_lemma_2} suggests the necessary rules for updating the generalized estimating sequences. 
	
	At this point, we note that the canonical structures for the terms in the sequences $\{\Phi_k (x)\}_{k=0}^\infty$ and $\{\psi_k (x)\}_{k=0}^\infty$ have not been introduced yet, and that Lemmas \ref{SFGM_lemma_2} and \ref{SFGM_lemma_3} hold for any construction of the generalized estimating sequences. These results stress on the generality of our proposed constructions for the newly introduced estimating sequences. Let us now present the constructions that will be used throughout the paper. 
	
	First, we define $\Phi_k (x) \triangleq \phi_k(x) - \psi_k(x)$, where $\phi_k (x) \triangleq \phi_k^* + \frac{\gamma_{k}}{2}||x - v_{k}||^2$ corresponds to the construction proposed in \cite[Lemma 2.2.3]{Nesterov_book}. We have already discussed that the function $\psi_k (x)$ can be selected in many ways. Since our goal here is to construct a generalized version of FGM which operates in a black-box setup, the simplest and quite generic approach to designing $\psi_k (x)$ is to accumulate the history of the previously constructed estimating functions. Therefore, we can define $\psi_k(x)$ as 
	\begin{align}
		\label{psii}
		\psi_k(x) \triangleq \sum_{i = 0}^{k-1} \beta_{i,k} \frac{\gamma_{i}}{2} ||x - v_{i}||^2, \; \forall k = 0, 1, \ldots. 
	\end{align}
	where $\beta_{i,k} \in \mathcal{R}, \forall i = 0, \ldots, k-1$ are weights assigned to each of the previously constructed scanning functions. Hereafter, we refer to $\psi_k (x)$ in \eqref{psii} as a \textit{heavy-ball type of momentum term}. Note that we allow the coefficients $\beta_{i,k}$ to change dynamically across the iterations. The intuition behind this choice follows from the fact that in black-box optimization no prior information on the function is available. Thus, the simplest thing to do is to let the terms in the sequence $\{\Phi_k (x)\}_{k=0}^{\infty}$ ``self-regulate''. Indeed, as the algorithm iterates towards optimality, several scanning functions are constructed. The accumulation of the information contained in the scanning functions is then captured by our model defined in \eqref{psii}. This also allows for defining $\psi_k (x)$ as a momentum term (or a ``heavy ball'') that is not directly applied to the iterates, but to the scanning function. As we will see later, this allows for better control of the parameters of $\Phi_k(x)$. From this perspective, the canonical structure of the new scanning function becomes
	\begin{align}
		\Phi_{k}(x) = \phi_{k}^* + \frac{\gamma_{k}}{2}||x - v_{k}||^2 -  \sum_{i = 0}^{k-1} \beta_{i,k} \frac{\gamma_{i}}{2} ||x - v_{i}||^2, \; \forall k. \label{scan_funct_k_SFGM} 
	\end{align}
	
	Note that we will rigorously establish later that the canonical structure for $\Phi_k (x)$ presented in \eqref{scan_funct_k_SFGM} is preserved by the recursive definition introduced in \eqref{phi_k+1_SFGM}. For now, let us observe that at iteration $k = 0$, \eqref{scan_funct_k_SFGM} is the same as the construction used for FGM. Afterwards, the memory term will begin to affect all the coefficients. From this perspective, a natural question to ask is: ``How large can the term $\sum_{i = 0}^{k-1} \beta_{i,k} \frac{\gamma_{i}}{2} ||x - v_{i}||^2$ become?" To answer this question, we note that the simplest way to guarantee that the necessary condition for Lemma \ref{SFGM_lemma_1} holds, is to restrict $\Phi_{k} (x)$ to be convex $\forall k = 0, 1, \ldots$. Therefore, utilizing the second order condition of convexity, we must have $\nabla^2 \Phi_{k}(x) \geq 0$. This implies that: 
	\begin{align}
		\label{psi_bound}
		\sum_{i = 0}^{k-1} \beta_{i,k} \gamma_{i} \leq \gamma_{k}.
	\end{align}
	
	\noindent Furthermore, in \eqref{def_1}, we also restrict the difference of functions $f(x) - \psi_k (x)$ to be convex for all $k = 0, 1, \ldots$. Since both functions are (by assumption) differentiable, from the second-order condition of convexity, it is sufficient to ensure that $\nabla^2 \left(f(x) - \psi_k (x)\right) \geq 0$. This results in 
	\begin{align}
		\label{second_part}
		\sum_{i = 0}^{k-1} \beta_{i,k} \gamma_{i} \leq \mu.
	\end{align}
	Combining \eqref{second_part} with \eqref{psi_bound}, we reach
	\begin{align}
		\label{psi_bound_}
		\sum_{i = 0}^{k-1} \beta_{i,k} \gamma_{i} \leq \min \left(\gamma_{k}, \mu\right).
	\end{align}

	Let us now analyze the minimal values that the terms in the sequence $\{\Phi_k (x)\}_{k=0}^\infty$ can have. First, define $x_{\Phi_k}^* \triangleq \arg \min_x \Phi_k (x)$. Then, utilizing \eqref{scan_funct_k_SFGM} for all values $k = 0, 1, \ldots$, we can write 
	\begin{align}
		\label{tko}
		\Phi_k^* = \min_x \Phi_k(x)  = \phi_k^*  + \frac{\gamma_{k}}{2}||x_{\Phi_k}^* - v_{k}||^2 - \sum_{i = 0}^{k-1} \beta_{i,k} \frac{\gamma_{i}}{2} ||x_{\Phi_k}^* - v_{i}||^2 .
	\end{align}
	Note that the coefficients $\phi_{k}^*, \gamma_{k}$ and $v_{k}$ are unknown and need to be found. Thus, the following lemma is in order.
	\begin{lemma}
		\label{SFGM_lemma_3}
		Let the coefficients $\beta_{i,k}$ be selected in a way that \eqref{psi_bound_} is satisfied, and let $\Phi_{0} (x) = \phi_{0}^* + \frac{\gamma_{0}}{2} ||x - v_{0}||^2$. Then, the process defined in Lemma \ref{SFGM_lemma_2} preserves the quadratic canonical structure of the scanning function introduced in \eqref{scan_funct_k_SFGM}. Moreover, the sequences $\{\gamma_{k}\}_{k=0}^\infty$, $\{v_{k}\}_{k=0}^\infty$ and $\{\phi_{k}^*\}_{k=0}^\infty$ can be computed as given by
		\begin{align}
			\label{gamma_expr}
			\gamma_{k+1} &= (1-\alpha_k)\gamma_{k} + \alpha_k \left(\mu + \sum_{i = 0}^{k - 1} \beta_{i,k} \gamma_{i} \right), \\ 
			\label{v_value}
			v_{k+1}  &= \frac{1}{\gamma_{k+1}}\left((1-\alpha_k)\gamma_k v_{k} + \mu \alpha_k \left(y_k - \frac{1}{\mu}\nabla f(y_k) + \sum_{i = 0}^{k - 1} \frac{\beta_{i,k} \gamma_{i}}{\mu} v_{i}\right)\right), \\ \nonumber
			\phi_{k+1}^*  &= \alpha_k f(y_k)  +  (1-\alpha_k) \phi_{k}^* +  \frac{\alpha_{k} \gamma_{k}(1-\alpha_{k}) (\mu + \sum_{i = 1}^{k-1} \beta_{i,k} \gamma_{i})}{2\gamma_{k+1}} ||y_k - v_{k}||^2 \\ \nonumber &+ \frac{\alpha_k^3}{\gamma_{k+1}} \sum_{i = 0}^{k - 1} \beta_{i,k} \gamma_{i} ||v_{i} - y_k|| \; ||\nabla f(y_k)|| + (1-\alpha_k) \frac{\gamma_{k}}{2}||x_{\Phi_k}^* \! - \! v_{k}||^2 \\ \nonumber &- \frac{\alpha_k^2 ||\nabla f(y_k) ||^2}{2 \gamma_{k+1}} + \frac{\alpha_k (1-\alpha_k)\gamma_k}{\gamma_{k+1}} \left( \left(v_{k} - y_k\right)^T \nabla f(y_k) + \sum_{i = 0}^{k - 1} \beta_{i,k} \gamma_{i} ||y_k - v_{i}|| ||y_k - v_{k}||\right)  \\ &+ \! \alpha_{k} \! \sum_{i = 0}^{k-1} \! \frac{\beta_{i,k} \gamma_{i}}{2} ||y_k \! - \! v_i||^2 \! + \! \frac{\left(1 \! - \! \alpha_{k}\right)\alpha_k^2}{\gamma_{k+1}} \sum_{i = 0}^{k - 1} \beta_{i,k} \gamma_{i} (v_{i} \! - \! y_k) ^T \nabla f(y_k) \! + \! \sum_{i = 0}^{k-1} \! \beta_{i,k} \frac{\gamma_{i}}{2} ||x_{\Phi_k}^* \! - \! v_{i}||^2 \label{psi_{k+1}^*}.
		\end{align}
	\end{lemma} 
	\begin{proof}
		See Appendix \ref{Proof of Lemma 3}. 
	\end{proof}
	
	Now, we can utilize an inductive argument to construct the algorithm. Assume that at iteration $k$, we have 
	\begin{align}
		\Phi_{k}^* \! \stackrel{\eqref{tko}}{=} \! \phi_k^* \! + \! \frac{\gamma_{k}}{2}||x_{\Phi_k}^* \! - \! v_{k}||^2 \! - \! \sum_{i = 0}^{k-1} \! \beta_{i,k} \frac{\gamma_{i}}{2} ||x_{\Phi_k}^* \! - \! v_{i}||^2 \! \geq \! f(x_k).
	\end{align}
	Then, from Lemma~\ref{SFGM_lemma_3}, at iteration $k+1$ we obtain \eqref{sfgm_eq_6} shown at the next page. From \eqref{sfgm_eq_6}, utilizing the lower bound \eqref{lower_bound} on $f(x_k)$ we arrive to
	\begin{align}
		\nonumber
		\phi_{k+1}^* \!  &\geq \! \alpha_k f(y_k) \! + \! (1 \! - \! \alpha_k) f(x_k) \! + \! \frac{\alpha_{k} \gamma_{k}(1 \! - \! \alpha_{k}) (\mu \! + \! \sum_{i = 1}^{k-1} \beta_{i,k} \gamma_{i})}{2\gamma_{k+1}} ||y_k \! - \! v_{k}||^2 \\ \nonumber &+ \alpha_{k} \sum_{i = 0}^{k-1} \frac{\beta_{i,k} \gamma_{i}}{2} ||y_k - v_{i} ||^2 - \frac{\alpha_k^2}{2 \gamma_{k+1}} ||\nabla f(y_k) ||^2 + \sum_{i = 0}^{k-1} \beta_{i,k} \frac{\gamma_{i}}{2} ||x_{\Phi_k}^* - v_{i}||^2 \\ \nonumber &+   \frac{\left(1 - \alpha_{k}\right)\alpha_k^2}{\gamma_{k+1}} \sum_{i = 0}^{k - 1} \beta_{i,k} \gamma_{i} (v_{i} - y_k) ^T \nabla f(y_k) + \frac{\alpha_k^3}{\gamma_{k+1}} \sum_{i = 0}^{k - 1}  \beta_{i,k} \gamma_{i} ||v_{i} - y_k|| \; ||\nabla f(y_k)||  \\ &+\frac{\alpha_k (1-\alpha_k)\gamma_k}{\gamma_{k+1}} \left((v_{k} - y_k)^T\nabla f(y_k) + \sum_{i = 0}^{k - 1} \beta_{i,k} \gamma_{i} ||y_k - v_{i}|| \; ||y_k - v_{k}|| \right). \label{sfgm_eq_6} 
	\end{align}
	Substituting the lower bound \eqref{lower_bound} into \eqref{sfgm_eq_6}, we obtain
	\begin{align}
		\nonumber
		\phi_{k+1}^*  &\geq \alpha_k f(y_k) \! + \! (1 \! - \! \alpha_k) \left( \! f(y_k) \! + \!  \nabla f(y_k)^T (x_k \! - \! y_k) \! + \! \frac{\mu}{2} ||y_k \! - \! x_k||^2 \! \right) \! + \! \alpha_{k} \sum_{i = 0}^{k-1} \frac{\beta_{i,k} \gamma_{i}}{2} ||y_k - v_{i} ||^2 \\ \nonumber &- \frac{\alpha_k^2}{2 \gamma_{k+1}} ||\nabla f(y_k) ||^2 + \! \frac{\alpha_{k} \gamma_{k}(1 \! - \! \alpha_{k}) (\mu \! + \! \sum_{i = 1}^{k-1} \beta_{i,k} \gamma_{i})}{2\gamma_{k+1}} ||y_k \! - v_{k}||^2  \\ \nonumber &+ \frac{\alpha_k^3}{\gamma_{k+1}} \sum_{i = 0}^{k - 1} \beta_{i,k} \gamma_{i} ||v_{i} - y_k|| \; ||\nabla f(y_k)|| + \sum_{i = 0}^{k-1} \beta_{i,k} \frac{\gamma_{i}}{2} ||x_{\Phi_k}^* - v_{i}||^2 \\ \nonumber &+  \frac{\left(1 - \alpha_{k}\right)\alpha_k^2}{\gamma_{k+1}} \sum_{i = 0}^{k - 1} \beta_{i,k} \gamma_{i} (v_{i} - y_k) ^T \nabla f(y_k) \\ \label{sfgm_eq_7} &+ \frac{\alpha_k (1-\alpha_k)\gamma_k}{\gamma_{k+1}} \left( \! (v_{k} - y_k)^T\nabla f(y_k) + \sum_{i = 0}^{k - 1} \beta_{i,k} \gamma_{i} ||y_k - v_{i}|| \; ||y_k - v_{k}|| \! \right) \! .  
	\end{align}

	From \eqref{sfgm_eq_7}, we discard all the positive terms and relax the lower bound. This results in
	\begin{align}
		\nonumber
		\phi_{k+1}^* \!  &\geq \! f(y_k) \! + \! (1 \! - \! \alpha_k) \nabla f(y_k)^T \! (x_k \! - \! y_k) \! - \! \frac{\alpha_k^2}{2 \gamma_{k+1}} ||\nabla \! f(y_k) ||^2   + \! \frac{\alpha_k (1 \! - \! \alpha_k)\gamma_k}{\gamma_{k+1}}(v_{k} \! - \! y_k)^T \! \nabla f(y_k) \!  \\ &+ \sum_{i = 0}^{k-1} \frac{\beta_{i,k} \gamma_{i}}{2} ||x_{\Phi_k}^*  - v_{i}||^2 + \left(1 - \alpha_{k}\right) \frac{\alpha_k^2}{\gamma_{k+1}} \sum_{i = 0}^{k - 1} \beta_{i,k} \gamma_{i} (v_{i} - y_k) ^T \nabla f(y_k). \label{pick_y_seq_}
	\end{align}
	For Lemma \ref{SFGM_lemma_1} to be valid, we must guarantee that $\Phi_{k+1}^* \geq f(x_{k+1})$. Observe that by adding $\frac{\gamma_{k}}{2}||x_{\Phi_k}^* \! - \! v_{k}||^2$ to the left-hand side (LHS) of \eqref{pick_y_seq_}, we have 
	\begin{align}
		\nonumber
		\phi_{k+1}^* + \frac{\gamma_{k}}{2}||x_{\Phi_k}^* \! - \! v_{k}||^2 - \sum_{i = 0}^{k-1} \frac{\beta_{i,k} \gamma_{i}}{2} ||x_{\Phi_k}^* - v_{i}||^2 \stackrel{\eqref{tko}}{=} \Phi_{k+1}^*.
	\end{align}
	This yields 
	\begin{align}
		\nonumber
		\Phi_{k+1}^* \!  &\geq \! f(y_k) \! + \! (1 \! - \! \alpha_k) \nabla f(y_k)^T \! (x_k \! - \! y_k) \! - \! \frac{\alpha_k^2}{2 \gamma_{k+1}} ||\nabla \! f(y_k) ||^2  + \frac{\alpha_k (1-\alpha_k)\gamma_k}{\gamma_{k+1}}(v_{k} - y_k)^T\nabla f(y_k) \\ &+ \left(1 - \alpha_{k}\right) \frac{\alpha_k^2}{\gamma_{k+1}} \sum_{i = 0}^{k - 1} \beta_{i,k} \gamma_{i} (v_{i} - y_k) ^T \nabla f(y_k). 		\label{pick_y_seq}
	\end{align}
	Moreover, we remark that the term $f(x_{k+1})$ can be obtained from \eqref{pick_y_seq} in several ways. Here, we choose to relax the lower bound even further by using the following form 
	\begin{align}
		\label{fgm_eq_7}
		f(y_k) - \frac{1}{2 L} ||\nabla f(y_k) ||^2 \geq f(x_{k+1}),
	\end{align}
	which can be guaranteed by a simple gradient descent step on $y_k$, that is,
	\begin{align}
		x_{k+1} = y_k - h_k \nabla f(y_k),
	\end{align}
	where, as can be seen from \eqref{upper_bound}, it suffices to let $h_k = \frac{1}{L}$. Therefore, we can compute $\alpha_{k}$ to have $\frac{1}{2L}$ as the coefficient for $||\nabla f(y_k)||^2$ in \eqref{pick_y_seq}. This results in:
	\begin{align}
		\label{alpha_k_intuition}
		\alpha_k = \sqrt{\frac{\gamma_{k+1}}{L}}.
	\end{align}
	Then, utilizing the recursive relation for $\gamma_{k+1}$ given in \eqref{gamma_expr}, its value can be computed in closed form by solving the quadratic equation as 
	\begin{align}
		\label{alpha_k_SFGM}
		\alpha_{k} &= \frac{\left(\mu + \sum_{i = 1}^{k-1} \beta_{i,k} \gamma_{i} - \gamma_{k}\right)}{2L} + \frac{\sqrt{\left(\mu + \sum_{i = 1}^{k-1} \beta_{i,k} \gamma_{i} - \gamma_{k}\right)^2 + 4L \gamma_{k}}}{2L}.
	\end{align}
	Making the above-mentioned selection for $\alpha_{k}$, we can now re-write \eqref{pick_y_seq} as
	\begin{align}
		\label{compute_y_k}
		\begin{split}
			\Phi_{k+1}^* \! &\stackrel{}{\geq} \! f (x_{k+1}) + (1-\alpha_k)  \nabla f(y_k)^T \!  \left( (x_k - y_k) + \frac{\alpha_k \gamma_{k}}{\gamma_{k+1}} (v_k - y_k)  + \frac{\alpha_k^2}{\gamma_{k+1}} \sum_{i = 0}^{k - 1} \beta_{i,k} \gamma_{i} (v_{i} - y_k) \right).
		\end{split}
	\end{align} 
	
	From \eqref{compute_y_k}, we can observe an important result from the computational point of view. It is the fact that the sequence of points $\{y_k\}_{k=0}^\infty$ ``comes for free'', in the sense that the points can be computed without the need to query a first-order oracle at point $x_k$. To obtain the update rule for the sequence $\{y_k\}_{k=0}^\infty$ it suffices to let 
	\begin{align} 
		\nonumber 
		x_k - y_k + \frac{\alpha_k \gamma_{k}}{\gamma_{k+1}} (v_k - y_k) + \frac{\alpha_k^2}{\gamma_{k+1}} \sum_{i = 0}^{k - 1} \! \beta_{i,k} \gamma_{i} (v_{i} - y_k) &= 0,
	\end{align} 
	which yields
	\begin{align} \label{y_k}
		y_k &=  \frac{\gamma_{k+1} x_k + \alpha_k \gamma_{k} v_{k} + \alpha_k^2 \sum_{i = 0}^{k - 1} \beta_{i,k} \gamma_{i} v_{i} }{\gamma_{k+1} + \alpha_k \gamma_{k} + \alpha_k^2 \sum_{i = 0}^{k - 1} \beta_{i,k} \gamma_{i}}. 
	\end{align}
	The closed-form expression for the points $y_k$ obtained in \eqref{y_k} again highlights the benefits of utilizing the generalized estimating sequence construction. Notice that the result of FGM is preserved, and the other terms come up as coefficients of the term $\alpha_k^2$. If we set $\beta_{i,k} = 0,$ $\forall i = 0, 1, \ldots, k-1$, i.e., $\psi_k (x) = 0$, then FGM is recovered. 
	
	Assuming that the coefficients $\beta_{i,k}$ are selected to comply with \eqref{psi_bound_}, we come to Algorithm \ref{SFGM_algorithm}. Comparing our proposed method with \cite[(2.2.19)]{Nesterov_book}, we first note that the selection of the next iterate is done in the same way in both algorithms. The reason for this update stems from the fact that both methods use \eqref{fgm_eq_7} to compute $x_{k+1}$. Moreover, both methods can be utilized in conjunction with many stopping criteria, such as a bound on the maximum number of iterations, norm of the gradient, etc. A similar type of update rule is also applied for the terms $\alpha_k$ and $\gamma_k$. Evidently, in this case both methods reflect the different types of estimating sequences that were used in constructing them. The computation of the points $y_k$ shares the same structure in both algorithms. In Algorithm~\ref{SFGM_algorithm}, the extra terms contributed from the generalized estimating sequence come up as coefficients of $\alpha_{k}^2$. The extra terms also appear in the update rule for $v_{k+1}$. Lastly, we emphasize that if we set the term $\psi_k(x) = 0$, \; $\forall k = 0, 1, \ldots$, then Algorithm~\ref{SFGM_algorithm} reduces to the regular FGM. This is consistent with the fact that the estimating sequences utilized in constructing FGM are a special case of the generalized estimating sequences that we used in constructing Algorithm~\ref{SFGM_algorithm}.
	
	\begin{algorithm}
		\caption{Proposed Method}
		\label{SFGM_algorithm}
		\begin{algorithmic}
			\STATE{Choose $x_0 \in \mathcal{R}^n$, set $\gamma_{0} = 0$ and $v_0 = x_0$.}
			\WHILE{stopping criterion is not met}
			\STATE{Compute $\alpha_{k} \in [0, 1]$ as $
				\alpha_{k} \! = \! \frac{ \! \left(\mu + \sum_{i = 1}^{k-1} \beta_{i,k} \gamma_{i} - \gamma_{k} \! \right) + \sqrt{\left(\mu + \sum_{i = 1}^{k-1} \beta_{i,k} \gamma_{i} - \gamma_{k}\right)^2 + 4L \gamma_{k}}}{2L} \! .$}			
			\STATE{Set $\gamma_{k + 1} = (1-\alpha_k)\gamma_{k} + \alpha_k \left(\mu + \sum_{i = 0}^{k - 1} \beta_{i,k} \gamma_{i} \right).$} 			
			\STATE{Choose $y_k = \frac{\gamma_{k+1} x_k + \alpha_k \gamma_{k} v_{k} + \alpha_k^2 \sum_{i = 0}^{k - 1} \beta_{i,k} \gamma_{i} v_{i} }{\gamma_{k+1} + \alpha_k \gamma_{k} + \alpha_k^2 \sum_{i = 0}^{k - 1} \beta_{i,k} \gamma_{i}}.$}
			\STATE{Set $
				x_{k+1} = y_k - \frac{1}{L} \nabla f(y_k)$}
			\STATE{Set $v_{k+1} = \frac{1}{\gamma_{k+1}}\bigg((1-\alpha_k)\gamma_k v_{k}$ \\ $\quad \quad \quad \quad + \mu \alpha_k \! \left( \! y_k \! - \! \frac{1}{\mu}\nabla f(y_k) + \sum_{i = 0}^{k - 1} \frac{\beta_{i,k} \gamma_{i}}{\mu} v_{i} \! \right) \! \bigg) \! .$}
			\ENDWHILE
		\end{algorithmic} 
	\end{algorithm} 
	
	%%%%%%%%%%%%%%%%%%%%%%%%%%%%%%%%%%%%%%%%%%%%%%%%%%%%%%%%%%%%%%%%%%%%
	\section{Convergence analysis}
	\label{Convergence analysis}
	%%%%%%%%%%%%%%%%%%%%%%%%%%%%%%%%%%%%%%%%%%%%%%%%%%%%%%%%%%%%%%%%%%%%
	As can be anticipated from Lemma~\ref{SFGM_lemma_1}, the convergence rate of Algorithm~\ref{SFGM_algorithm} will depend on both the $\{\lambda_k\}_{k=0}^\infty$ and $\{\psi_k (x)\}_{k=0}^\infty$ sequences. The following theorem makes this statement precise and allows us to characterize the convergence rate of Algorithm~\ref{SFGM_algorithm}.
	\begin{theorem}
		\label{conv_analysis_t_1}
		Let $\lambda_0 = 1$ and $\lambda_k = \prod_{i = 0}^{k-1} \left(1 - \alpha_i \right)$. Then, Algorithm~\ref{SFGM_algorithm} generates a sequence of points $\{x_k\}_{k = 0}^\infty$ such that
		\begin{align}
			f(x_k) - f^* &\leq \lambda_{k} \left[ f(x_0) - f(x^*) + \frac{\gamma_{0}}{2}||x_0 - x^*||^2 \right] - (1 - \lambda_k)\psi_k(x^*).
		\end{align} 
		\begin{proof} 
			See Appendix \ref{Proof of Theorem 1}.
		\end{proof}
	\end{theorem}
	Comparing the result presented in Theorem~\ref{conv_analysis_t_1} to \cite[Theorem 2.2.1]{Nesterov_book}, we observe that as long as $\psi_k > 0$, we should expect Algorithm~\ref{SFGM_algorithm} to yield a faster convergence to optimality than the one exhibited by FGM. For this reason, we will refer to our proposed Algorithm~\ref{SFGM_algorithm} as SuperFGM (SFGM). 
	
	To analyse the rate of convergence, we start by computing the rate at which the sequence $\{\lambda_{k}\}_{k=0}^\infty$ decreases. The following lemma is in order. 
	\begin{lemma}
		\label{conv_analysis_lemma_1}
		For all $k \geq 0$, Algorithm~\ref{SFGM_algorithm} guarantees that
		\begin{align}
			\label{FGM_conv_eq_66} 
			\lambda_{k} \! &\leq \! \frac{2 \mu}{L \! \left( \! e^{\frac{k + 1}{2} \sqrt{\frac{\mu + \sum_{i = 1}^{k-1} \beta_{i,k} \gamma_i}{L}}} \! - e^{-\frac{k + 1}{2} \sqrt{\frac{\mu + \sum_{i = 1}^{k-1} \beta_{i,k} \gamma_i}{L}}} \! \right)^2} \leq \frac{2 \mu}{\left(\mu + \sum_{i = 1}^{k-1} \beta_{i,k} \gamma_i\right) (k+1)^2}.
		\end{align} 
	\end{lemma}
	\begin{proof}
		See Appendix \ref{Proof of Lemma 4}.
	\end{proof}

	Now, to show that SFGM is also an optimal method, the following theorem is in order.
	\begin{theorem}
		\label{tttttttttttttttttttttttt}
		In Algorithm~\ref{SFGM_algorithm}, let $\mu > 0$. Then, the algorithm generates a sequence of points such that
		\begin{align}
			\label{tttr}
			f(x_k) \! - \! f(x^*) \! &\leq \! \! \frac{\mu ||x_0 - x^*||^2}{\left( \! \! e^{ \! \frac{k + 1}{2} \! \sqrt{\frac{ \! \mu + \! \sum_{i = 1}^{k-1} \beta_{i,k} \gamma_{i}}{L}}} \! \! - \! e^{ \! -\frac{k + 1}{2} \! \sqrt{\frac{ \! \mu + \! \sum_{i = 1}^{k-1} \beta_{i,k} \gamma_{i}}{L}}} \! \right)^2} - (1-\lambda_k)\psi_k (x^*) %\\ \label{ttr} & \leq \! \frac{\mu L ||x_0 \! - \! x^*||^2}{(\mu + \sum_{i = 1}^{k-1} \beta_{i,k} \gamma_{i}) (k\! + \! 1)^2}
			%\leq \! \frac{2 \mu L ||x_0 \! - \! x^*||^2}{\left(\! \mu \! + \! \sum_{i = 1}^{k-1} \! \beta_{i,k} \gamma_{i} \! \right) \! \! (k\! + \! 1)^2} \! - \! (1 \! - \! \lambda_k)\psi_k (x^*)
			.
		\end{align}
		This means that the method is optimal when the accuracy $\epsilon$ is small enough, that is,  
		\begin{align}
			\label{eps}
			\epsilon \leq \frac{\mu}{2} R_0^2.
		\end{align}
	\end{theorem}
	\begin{proof}
		See Appendix \ref{Proof of Theorem 2}. 
	\end{proof}
	
	Finally, we can directly compare our proposed method to FGM, which requires the following number of iterations \cite[(2.2.17)]{Nesterov_book}
	\begin{align}
		\label{k_fgm}
		k_{FGM} &\geq \sqrt{\frac{L}{\mu}} \left(\text{ln} \left(\frac{\mu R_0^2}{2 \epsilon}\right) + \text{ln}(23/3)\right).
	\end{align}
	Comparing the bound in \eqref{k_fgm} to the bound obtained from our proposed method in \eqref{lower_bound__} (see the supplemental materials), we can observe that SFGM always outperforms FGM despite any valid selection of the coefficients $\beta_{i,k}$. Under the selection $\beta_{i,k} = 0, \forall i= 1, \ldots k-1$, which reduces SFGM to FGM, we observe that we still have an improvement of a constant number of iterations. This stems from the fact that our result obtained in Lemma~\ref{conv_analysis_lemma_1} yields a tighter bound on the sequence $\{\lambda_k\}_{k=0}^\infty$. Moreover, it also supports the smallest possible starting value for initializing the sequence $\{\gamma_k\}_{k=0}^\infty$, which is $\gamma_{0} = 0$, which is not supported by the existing analysis for FGM.  
	
	Allowing for nonzero values of $\beta_{i,k}$, a better scaling factor than for FGM is also obtained. Moreover, note that the bound obtained in \eqref{sfgm_bound} (see supplemental materials) is dynamic, and if $\sum_{i = 1}^{k-1} \beta_{i,k} \gamma_{i} \rightarrow \mu$, then we obtain the tightest provable bound on the performance of SFGM. Here, we remark that \eqref{sfgm_bound} is still an upper bound on the true performance of SFGM. The reason for that is that it is based on the bound obtained in \eqref{FGM_conv_eq_7} (see supplemental materials), which does not account for the extra terms coming from the sequence $\{\psi_k (x)\}_{k=0}^\infty$. The rationale behind this approach stems from the difficulty of estimating the size of the terms in the sequence $\{\psi_k (x)\}_{k=0}^\infty$.
	
	So far, no explicit construction about the terms $\beta_{i,k}$ has been given. Evidently, they act as weights that allow us to parse function information. From the result of Lemma~\ref{conv_analysis_lemma_1}, we observe that it is beneficial to allow the term $\sum_{i = 1}^{k-1} \beta_{i,k} \gamma_{i}$ to be as large as possible. The bound for this term has been obtained in \eqref{psi_bound_}. There are several ways to select the coefficients $\beta_{i,k}$, $\forall i, k$, and at the same time satisfy the bound. For instance, $\beta_{i,k}$ can be selected to account for certain samples of the previously constructed scanning functions, or a window of the previous scanning functions, or it can act as a forgetting factor that spans the entire range of the scanning functions with some weight. For this paper, we pose the optimal selection of these coefficients as an open problem, and focus on the simplest choice for the coefficients, that is,  
	\begin{align}
		\label{betas_selection}
		\beta_{i,k} =   \begin{cases}
			\min \left(1, \frac{\mu}{\gamma_{k-1}} \right),       & \quad \text{if } i = k-1, \\
			0,  & \quad \text{otherwise}.
		\end{cases}
	\end{align}
	
	With this selection of $\beta_{i,k}$, the lower bound on the number of iterations becomes
	\begin{align}
		\label{new_sfgm_bound}
		k_{SFGM} &\geq \sqrt{\frac{L}{\mu + \min \left(\gamma_{k - 1}, \mu\right)}} \left(\text{ln} \left(\frac{\mu R_0^2}{2 \epsilon}\right) + \text{ln}(5)\right).
	\end{align} 
	Then, from \eqref{psi_bound_}, \eqref{gamma_expr} and allowing $k \rightarrow \infty$, we obtain
	\begin{align}
		\label{asymptotic_bound}
		k_{SFGM} \rightarrow \sqrt{\frac{L}{2 \mu}} \left(\text{ln} \left(\frac{\mu R_0^2}{2 \epsilon}\right) + \text{ln}(5)\right).
	\end{align}
	Let us now analyze the relative behavior of terms $\alpha_k$ and $\gamma_k$. From the update rule of the sequence $\{\alpha_k\}_{k=0}^\infty$, \eqref{alpha_k_intuition}, we can observe that $\alpha_{k} \propto \gamma_{k + 1}$. Therefore, if the value of $\gamma_{k+1}$ increases, the value of $\alpha_{k}$ also increases. From the relationship for computing $\gamma_{k + 1}$ obtained in \eqref{gamma_expr}, we can see that also $\gamma_{k + 1}$ increases with $\alpha_{k}$. Therefore, we can conclude that these two terms recursively increase the value of one-another. In Lemma~\ref{SFGM_lemma_1}, we established that $\gamma_{0} = 0$. Then, from the update rule of the sequence $\{\gamma_k\}_{k=0}^\infty$, we can see that $\gamma_{1} > \gamma_{0}$. This results in a value of $\alpha_{0} > 0$, which then causes the values of the sequence $\{\gamma_k\}_{k=0}^\infty$ to increase. Therefore, as the algorithm progresses and the values of $\gamma_{k - 1}$ increase, the bound in \eqref{new_sfgm_bound} converges to \eqref{asymptotic_bound}. Lastly, we emphasize that the LHS in \eqref{new_sfgm_bound} converges to \eqref{asymptotic_bound} very quickly due to the exponential growth of the terms in the sequence $\{\gamma_k\}_{k=0}^\infty$. Analytically, this can be seen by writing $\alpha_k = \sqrt{\gamma_{k+1}/L} = 1 - \lambda_{k+1}/\lambda_k$, and observing from Lemma~\ref{conv_analysis_lemma_1} that the terms of the sequence $\{\lambda_k\}_{k=0}^\infty$ decrease exponentially. Numerically, this is also shown in Section~\ref{linear_regression}.

	%%%%%%%%%%%%%%%%%%%%%%%%%%%%%%%%%%%%%%%%%%%%%%%%%%%%%%%%%%%%%%%%%%%%
	\section{Numerical study}
	\label{numericals}
	%%%%%%%%%%%%%%%%%%%%%%%%%%%%%%%%%%%%%%%%%%%%%%%%%%%%%%%%%%%%%%%%%%%%
	In this section, we test the efficiency of several instances of the proposed method both in terms of decreasing the distance to optimality, as well as in decreasing the norm of the gradient. Motivated by different types of applications in statistical signal processing, machine learning, inverse problems, etc., we focus on minimizing the quadratic and the logistic loss functions. Both synthetic and real data are utilized to analyze different aspects of the algorithm. The synthetic data, which are randomly generated, are used to have a better insight on how the performance of the methods scales with the condition number of the problem. On the other hand, the real-world datasets are drawn from the Library for Support Vector Machines (LIBSVM) \cite{LIBSVM}. The datasets that we use are selected according to the specific problem instances. For comparison purposes, we also utilize CVX~\cite{CVX} to find the optimal solutions. 
	%Our computational results assert that the theoretical gains, which were established in the earlier sections, also carry through when evaluating the practical performance of the methods.
	
	We benchmark against two instances of FGM Constant Step Scheme I (CSS1). More specifically, we consider the starting values for $\gamma_{0} = L$, which we refer to as FGM CSS1, and $\gamma_{0} = \mu$, which yields the best performance for FGM. The latter also corresponds to Constant Step Scheme~III (CSS3) \cite[Chapter 2.2]{Nesterov_book}. To simulate SFGM, we consider the simplest instances of the algorithm, respectively selecting $\beta_{0, k} = 1$ and $\beta_{i,k} = 0, \forall i = 1, \ldots k$. This instance of the algorithm is referred to as memoryless SFGM. We note that when $\gamma_{0} = 0$, this algorithm corresponds to FGM. However, the original analysis of FGM does not guarantee convergence of the method with $\gamma_{0} = 0$, whereas SFGM guarantees convergence, and achieves it in a smaller number of iterations. The other instance of the algorithm that is considered, is the one introduced in \eqref{betas_selection}. This instance is referred to as SFGM with memory term $\gamma_{k - 1}$. Relative to the CSS1 of FGM, this instance of SFGM requires the storage of an extra vector and scalar. Regarding the computations, it performs four more scalar additions and one more vector addition. Nevertheless, despite this slight increase in computational burden, we have already proved that SFGM with memory term $\gamma_{k-1}$ is an optimal method. Lastly, the starting point $x_0$ is randomly selected and all algorithms are initiated in it.
	
	%%%%%%%%%%%%%%%%%%%%%%%%%%%%%%%%%%%%%%%%%%%%%%%%%%%%%%%%%%%%%%%%%%%%
	\subsection{Decreasing the distance to optimality}
	\label{linear_regression}
	%%%%%%%%%%%%%%%%%%%%%%%%%%%%%%%%%%%%%%%%%%%%%%%%%%%%%%%%%%%%%%%%%%%%
	
	We start by solving problems of the form
	\begin{equation}
		\begin{aligned}
			\label{mse}
			& \underset{x \in \mathcal{R}^n}{\text{minimize}}
			& &\frac{1}{2m} \sum\limits_{i=1}^m (a_i^T x - y_i)^2 + \frac{\tau}{2} ||x||^2.  
		\end{aligned}
	\end{equation}
	The main goal of this section is to show that the theoretical convergence guarantees obtained in Section~\ref{Convergence analysis} yield a realistic description of the practical performance of the methods. Moreover, we analyze how the performance of the methods scales with the condition number of the problem. We also show the fast convergence of the terms in the sequence $\{\gamma_k\}_{k=0}^\infty$.
	
	Let us begin by considering the simplest case, $\tau = 0$. To generate the data, we consider a symmetric positive definite diagonal matrix $A \in \mathcal{R}^{m\text{x}m}$, whose elements $a_{ii}$ are drawn from the discrete set $\{10^0, 10^{-1}, 10^{-2}, \ldots 10^{-\xi} \}$ uniformly at random. This ensures control over the condition number of the matrix $A$, which will be $10^{\xi}$. Moreover, this choice of constructing $A$ yields the values for $L = 1$ and $\mu = 10^{-\xi}$. The entries of the vector $y \in \mathcal{R}^m$ are uniformly drawn from the box $[0, 1]^n$. In our computational experiments, we set $m = 1000$ and $\xi \in \{3, 4\}$. Our findings are reported in Fig.~\ref{num_sec_fig_1}. 
	\begin{figure} 
		\centering
		\begin{subfigure}[h]{0.49\columnwidth}
			\centering
			\includegraphics[width=1\columnwidth,height = 0.8\linewidth, trim={3cm 9.5cm 3cm 7.5cm}]{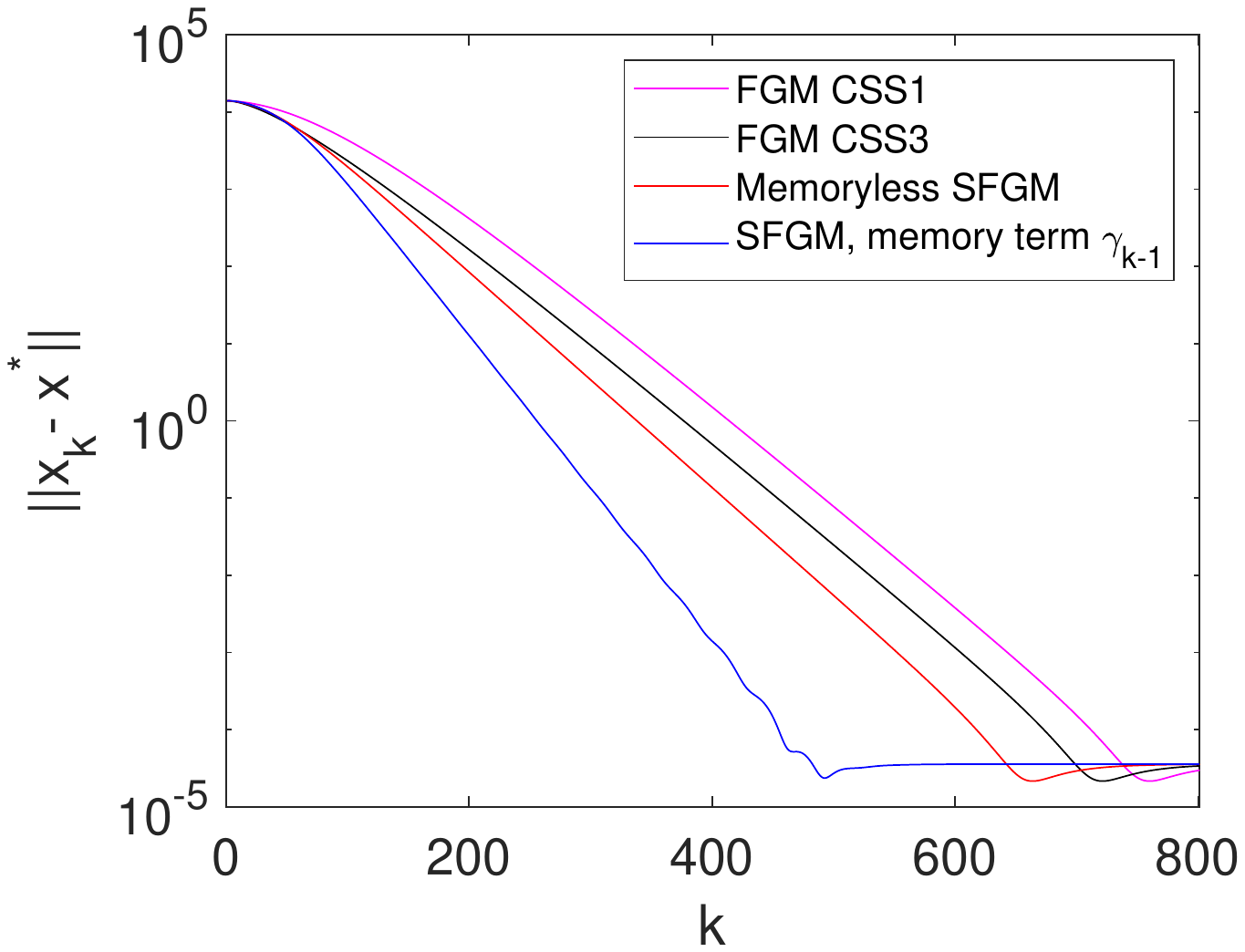} 		\label{num_sec_fig_1_sub_1}  \vspace{-10mm}
			\caption{Decreasing the distance to $x^*$, $\kappa = 10^{3}$.} \vspace{-10mm}
		\end{subfigure}
		% \vfill
		\begin{subfigure}[h]{0.49\columnwidth}
			\centering
			\includegraphics[width=\columnwidth,height = 0.8\linewidth, trim={3cm 9.5cm 3cm 7.5cm}]{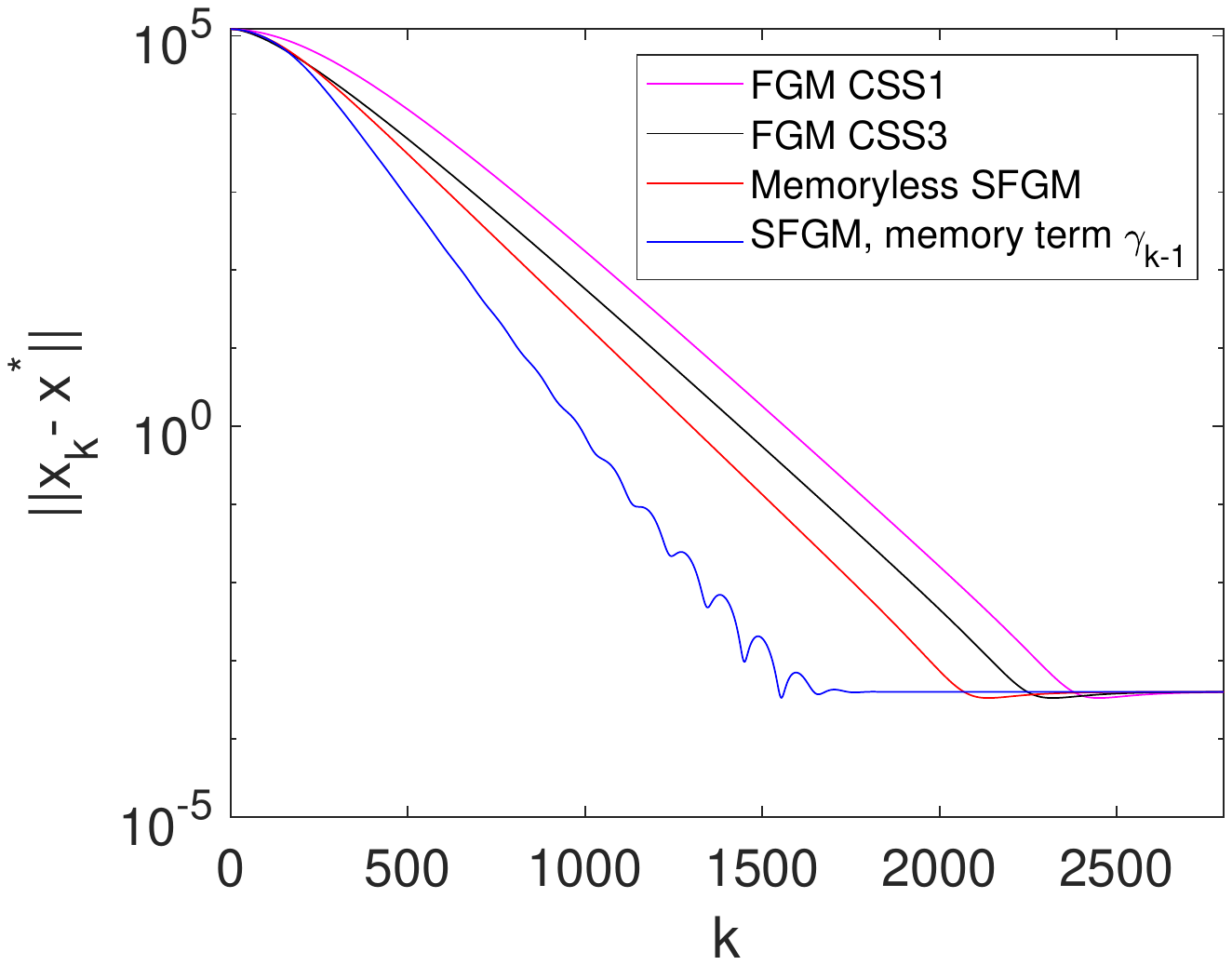} 		\label{num_sec_fig_1_sub_2}  \vspace{-10mm}
			\caption{Decreasing the distance to $x^*$, $\kappa = 10^{4}$.}  \vspace{-10mm}
		\end{subfigure}
		\begin{subfigure}[h]{0.49\columnwidth}
			\centering
			\includegraphics[width = \columnwidth,height = 0.8\linewidth, trim={3cm 9.5cm 3cm 7.5cm}]{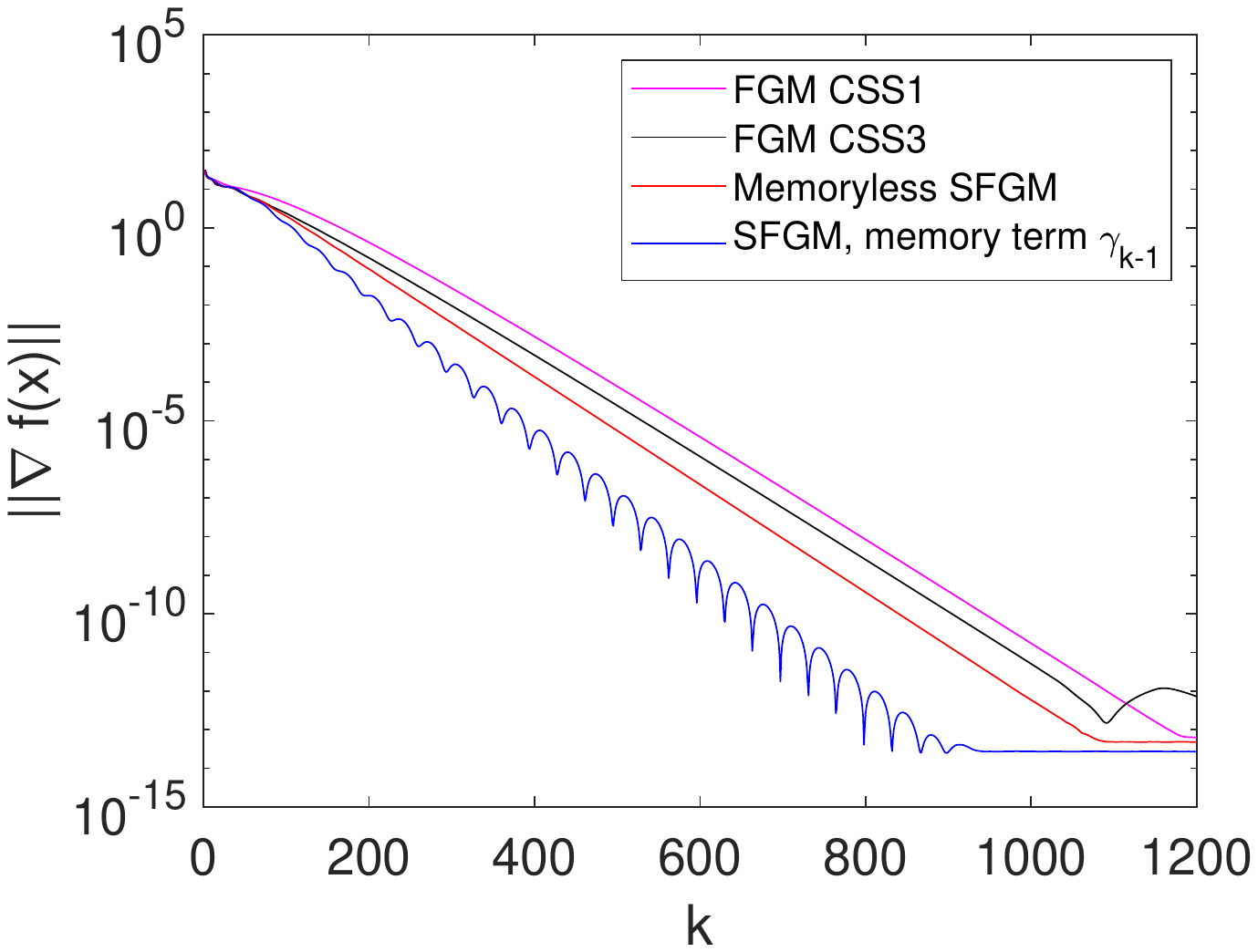} 		\label{num_sec_fig_1_sub_3}  \vspace{-10mm}
			\caption{Decreasing the norm of the gradient, $\kappa = 10^{3}$.} \vspace{-10mm}
		\end{subfigure}
		% \vfill
		\begin{subfigure}[h]{0.49\columnwidth}
			\centering
			\includegraphics[width=\columnwidth,height = 0.8\linewidth, trim={3cm 9.5cm 3cm 7.5cm}]{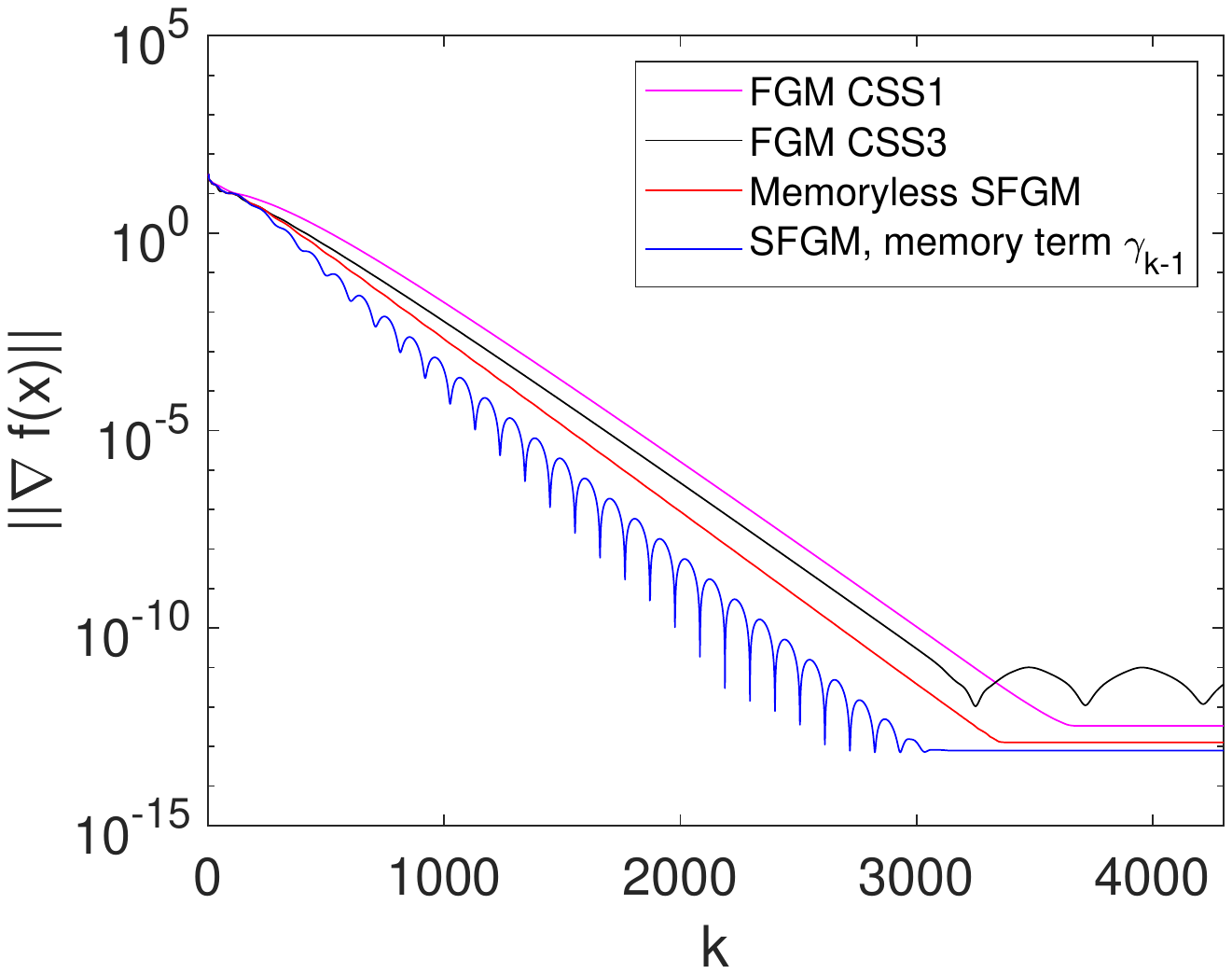} 		\label{num_sec_fig_1_sub_4} \vspace{-10mm}
			\caption{Decreasing the norm of the gradient, $\kappa = 10^{4}$.}  \vspace{-10mm}
		\end{subfigure}
		\begin{subfigure}[h]{0.49\columnwidth}
			\centering
			\includegraphics[width = \columnwidth,height = 0.8\linewidth, trim={3cm 9.5cm 3cm 7.5cm}]{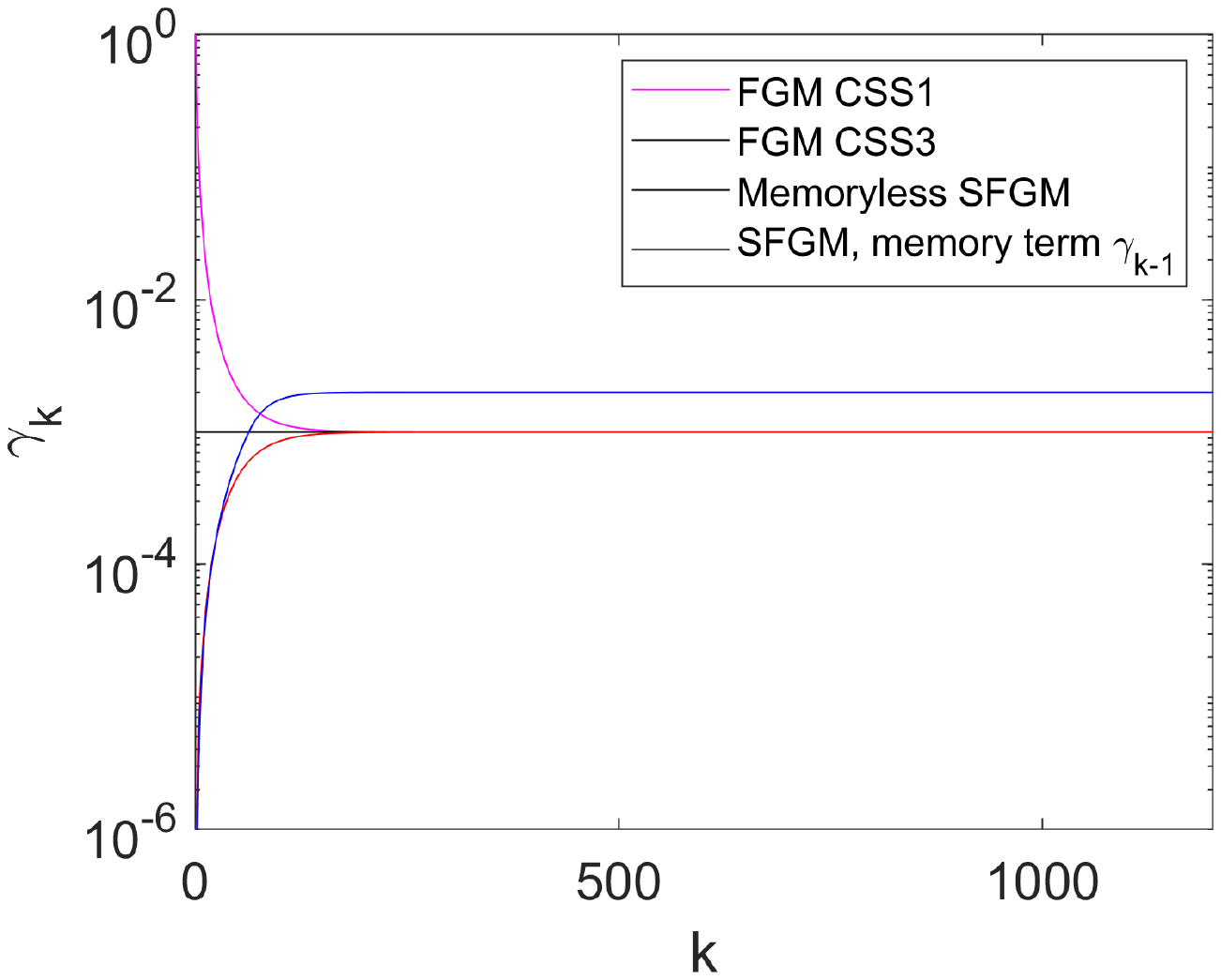} 		\label{num_sec_fig_1_sub_5}  \vspace{-10mm}
			\caption{Convergence of the sequence $\{\gamma_k\}_{k=0}^\infty$, $\kappa = 10^{3}$.} 
		\end{subfigure}
		\begin{subfigure}[h]{0.49\columnwidth}
			\centering
			\includegraphics[width= \columnwidth,height = 0.8\linewidth, trim={3cm 9.5cm 3cm 7.5cm}]{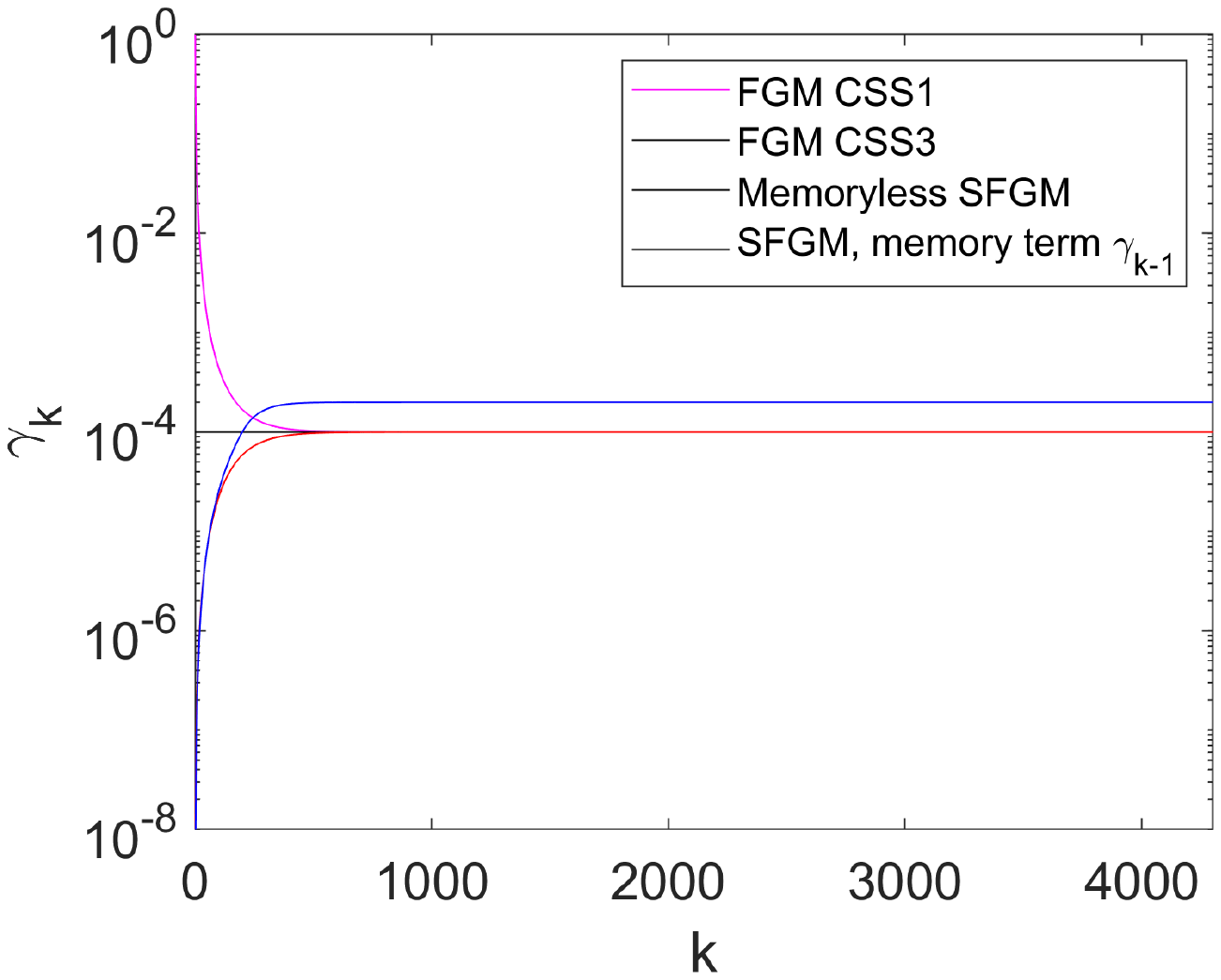} 		\label{num_sec_fig_1_sub_6} \vspace{-10mm}
			\caption{Convergence of the sequence $\{\gamma_k\}_{k=0}^\infty$, $\kappa = 10^{4}$.} 
		\end{subfigure}
		\caption{Comparison between various features of interest of the tested algorithms. The goal is to minimize the quadratic loss function, for which $A \in \mathcal{R}^{1000 \text{x} 1000}$ and its entries are randomly generated.}
		\label{num_sec_fig_1}
	\end{figure}
	
	From Fig.~\ref{num_sec_fig_1}, we can see the performance gains of SFGM. The quality of the obtained solution, as measured by the distance to the optimal solution $x^*$, is similar to that obtained by FGM, however the number of iterations required by SFGM is smaller. In the case of the memoryless version of SFGM, we can observe that it exhibits the same behavior as FGM, however it converges faster. This is coherent with the theoretical bounds established in Section~\ref{Convergence analysis}. A similar observation can also be made for the case of SFGM with memory term $\gamma_{k-1}$. From Figs.~\ref{num_sec_fig_1}(a)~and~\ref{num_sec_fig_1}(b), we can see that the method yields an improvement of approximately $30\%$ over FGM CSS3. This result is also coherent with the theoretical asymptotic bound obtained in \eqref{asymptotic_bound}, that also suggests an improvement of $30\%$ over FGM. A similar observation can also be made from Figs.~\ref{num_sec_fig_1}(c)~and~\ref{num_sec_fig_1}(d), in which we report the decrease in the norm of the gradient. Moreover, from Figs.~\ref{num_sec_fig_1}(e)~and~\ref{num_sec_fig_1}(f), we can observe the exponential convergence of the term $\gamma_{k-1}$ to $\mu$. Lastly, as the condition number of the problem increases, all methods require a larger number of iterations to converge. For instance, from Fig.~\ref{num_sec_fig_1}(a), we can see that when $\kappa = 10^3$ the performance difference between the tested algorithms is of the order of hundreds of iterations. Then, when $\kappa = 10^4$, from Fig.~\ref{num_sec_fig_1}(b), we can see that the differences between algorithms increases. In the sequel, we will see that for more ill conditioned problems, the differences between the tested algorithms become even larger. 
	
	Next, we proceed by considering the more general case, $\tau \neq 0$. We let $A \in \mathcal{R}^{m \text{x}n}$ and $b \in \mathcal{R}^{m}$ and start with the case when $m < n$. Both synthetic and real data are utilized. To diversify the type of synthetic data used, here we do not impose any particular structure on $A$. We simply draw the elements for both $A$ and $b$ from a standard normal distribution and set $m = 800$ and $n = 1000$. Regarding real data, we utilize the ``colon-cancer'' dataset, for which $m = 62$ and $n = 2000$. The data that is used also dictates the values of $L$ and $\mu$. In practice, estimating $\mu$ is challenging and computationally expensive. For this reason, the common approach that is followed is to assume that the strong convexity parameter of the data is $0$. In all the numerical experiments that will be presented in the sequel, we also follow this approach, and equate $\mu$ to the regularization parameter $\frac{\tau}{2}$. On the other hand, similar to the previous computational experiments (and to be coherent with the theoretical analysis) we estimate the Lipschitz constant directly from the data. Nevertheless, we note that several efficient backtracking strategies for estimating $L$ already exist in the literature \cite{Nesterov_2007, Tseng}. For the datasets that we are utilizing, the respective Lipschitz constants are $L_{\text{``random''}} = 3567.1$ and $L_{\text{``colon-cancer''}} = 1927.4$. Moreover, for both data types, we let the regularizer term $\tau \in  \{10^{-5}, 10^{-6}\}$. Evidently, this selection of the regularizer term ensures that the condition number of the problems that are being solved is quite high. The numerical results are presented in Fig.~\ref{num_sec_fig_2}, from which we can observe that SFGM with memory term $\gamma_{k-1}$ again outperforms FGM CSS3 by approximately $35\% -40\%$. 
	\begin{figure}
		\centering
		\begin{subfigure}[h]{0.49\columnwidth}
			\centering
			\includegraphics[width=1\columnwidth,height = \linewidth, trim={3cm 9.5cm 3cm 7.5cm}]{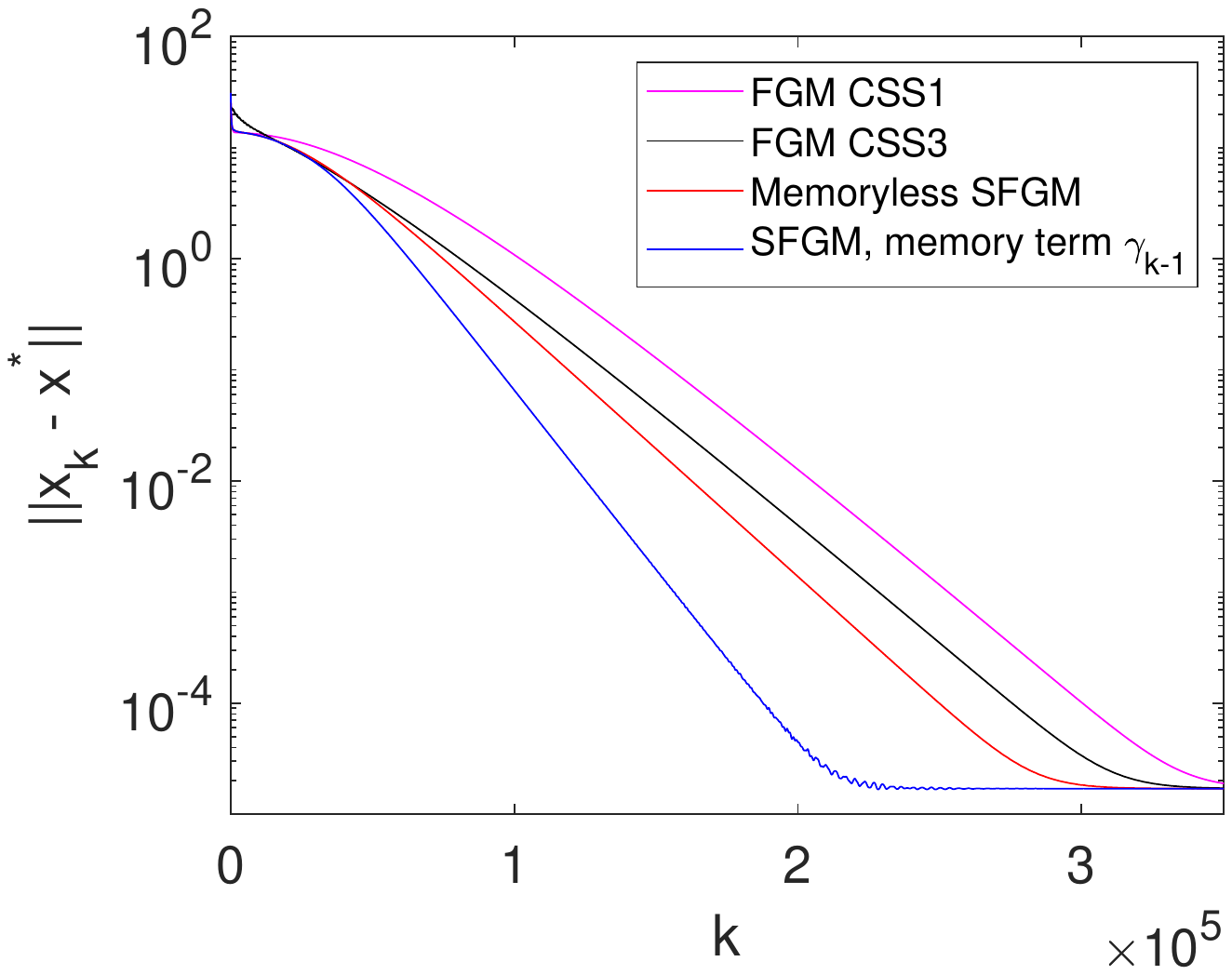} 		\label{num_sec_fig_2_sub_1} \vspace{-10mm}
			\caption{Random data, $\tau = 10^{-5}$.} \vspace{-10mm}
		\end{subfigure}
		% \vfill
		\begin{subfigure}[h]{0.49\columnwidth}
			\centering
			\includegraphics[width=\columnwidth,height = \linewidth, trim={3cm 9.5cm 3cm 7.5cm}]{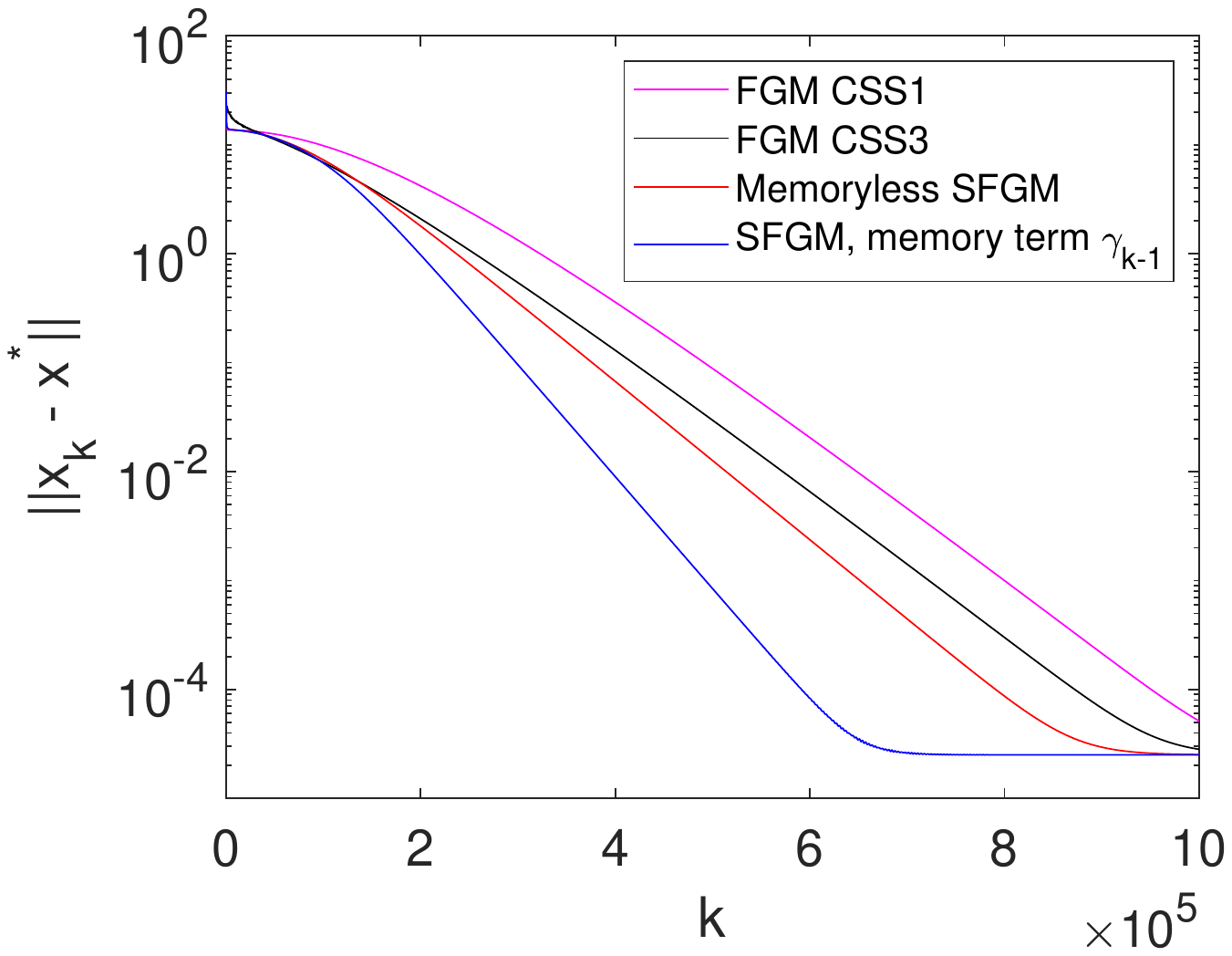} 		\label{num_sec_fig_2_sub_2}  \vspace{-10mm}
			\caption{Random data, $\tau = 10^{-6}$.} \vspace{-10mm}
		\end{subfigure}
		\begin{subfigure}[h]{0.49\columnwidth}
			\centering
			\includegraphics[width = \columnwidth,height = \linewidth, trim={3cm 9.5cm 3cm 7.5cm}]{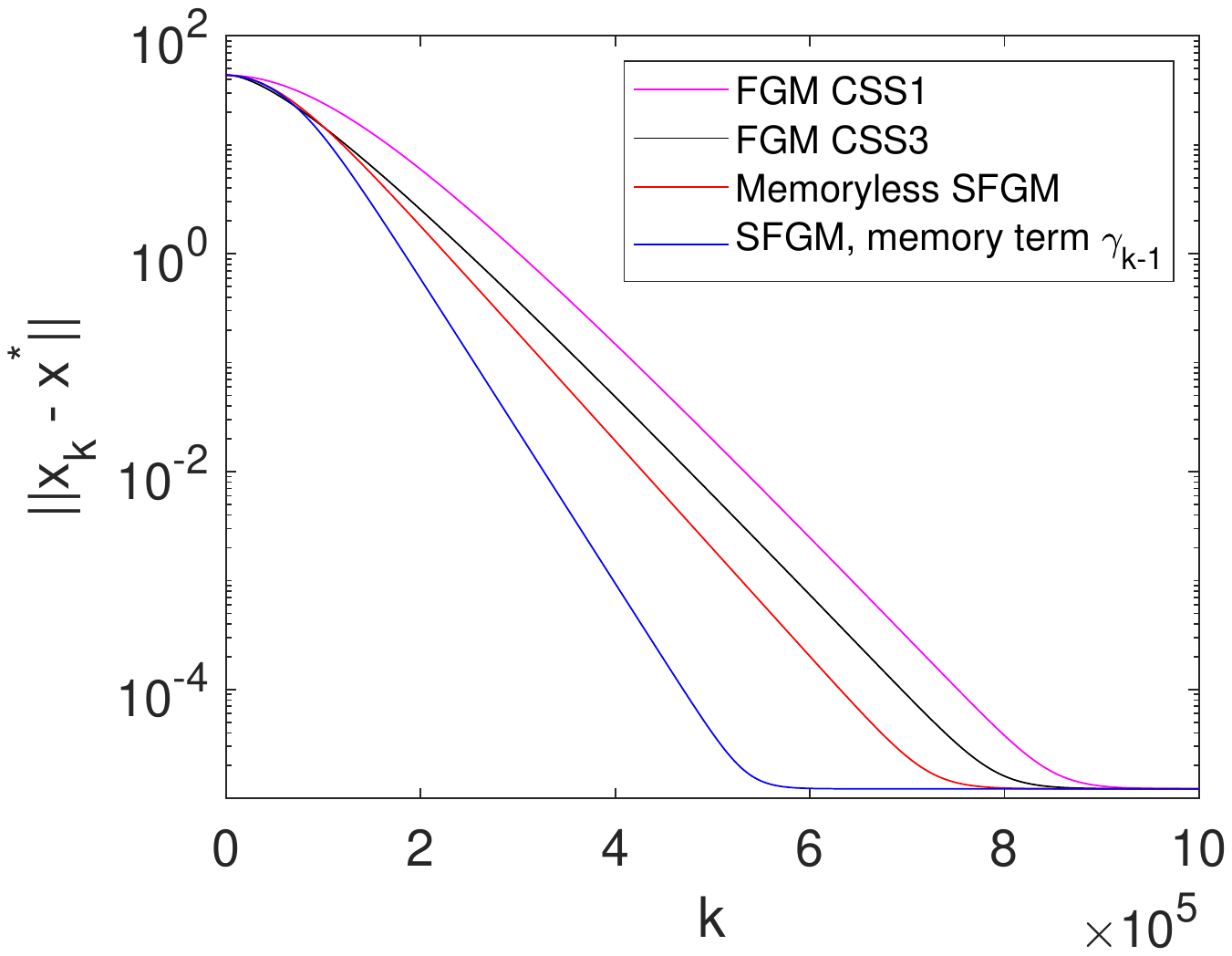} 		\label{num_sec_fig_2_sub_5} \vspace{-8mm}
			\caption{The dataset is ``colon-cancer'', $\tau = 10^{-5}$.}
		\end{subfigure}
		% \vfill
		\begin{subfigure}[h]{0.49\columnwidth}
			\centering
			\includegraphics[width=\columnwidth,height = \linewidth, trim={3cm 9.5cm 3cm 7.5cm}]{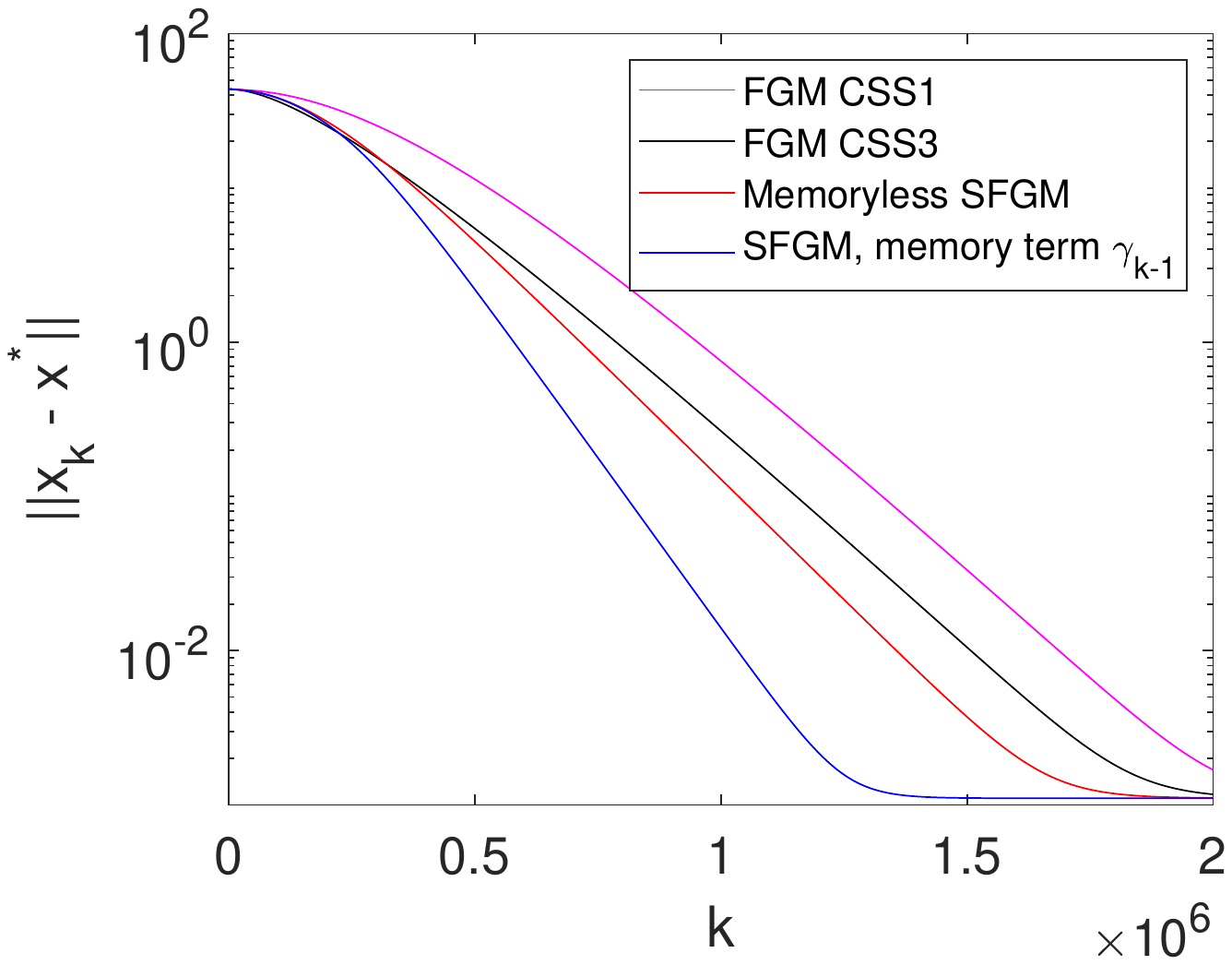} 		\label{num_sec_fig_2_sub_4} \vspace{-8mm}
			\caption{The dataset is ``colon-cancer'', $\tau = 10^{-6}$.}  
		\end{subfigure}
		\caption{Comparison between the efficiency of algorithms tested in minimizing the regularized quadratic loss function in the case where $m < n$, i.e., $A$ is a fat matrix.}
		\label{num_sec_fig_2}
	\end{figure}
	
	Finally, we analyze the remaining case, in which the matrix $A$ is a tall matrix. For this, we only consider real data. The datasets that we selected were ``triazine'' and ``a1a''. For the former dataset, we have $m = 186$ and $n = 60$. For the latter, we have $m = 1605$  and $n = 123$. The corresponding Lipschitz constants are $L_{\text{``triazines''}} = 632.2804$ and $L_{\text{``a1a''}} = 10061$. The regularizer term is set $\tau \in \{10^{-7}, 10^{-8}\}$. The results are reported in Fig.~\ref{num_sec_fig_3}. Despite the fact that the problems being solved are extremely ill-conditioned, we can see that the fastest version of SFGM retains its theoretical gains of approximately $30\% - 35\%$ across all datasets, when compared to the fastest version of FGM, which is CSS3. 
	\begin{figure}
		\centering
		\begin{subfigure}[h]{0.49\columnwidth}
			\centering
			\includegraphics[width=1\columnwidth,height = \linewidth, trim={3cm 9.5cm 3cm 8cm}]{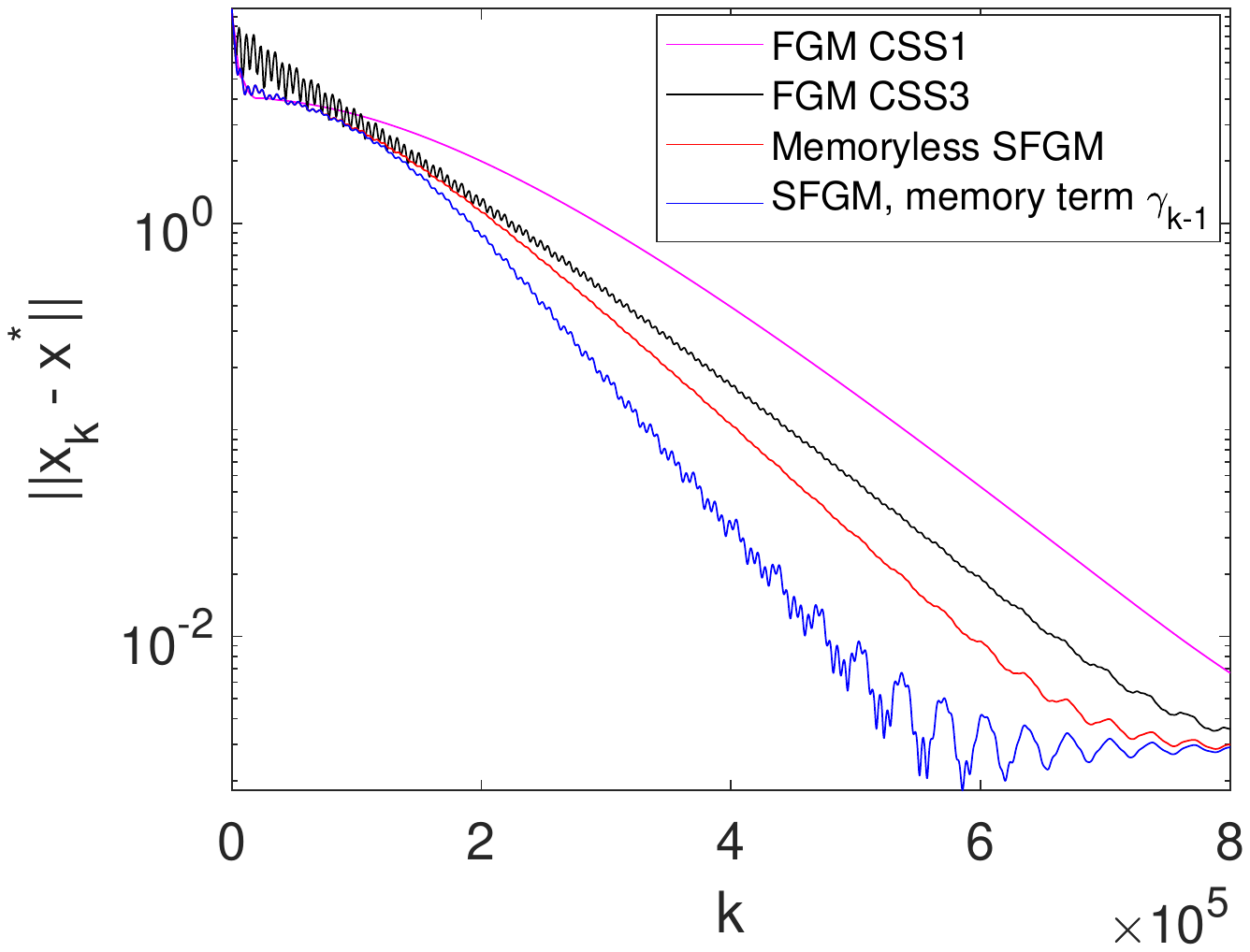} 		\label{num_sec_fig_3_sub_1} \vspace{-8mm}
			\caption{The dataset is ``triazines'' and $\tau = 10^{-7}$.} \vspace{-8mm}
		\end{subfigure}
		% \vfill
		\begin{subfigure}[h]{0.49\columnwidth}
			\centering
			\includegraphics[width=\columnwidth,height = \linewidth, trim={3cm 9.5cm 3cm 7.5cm}]{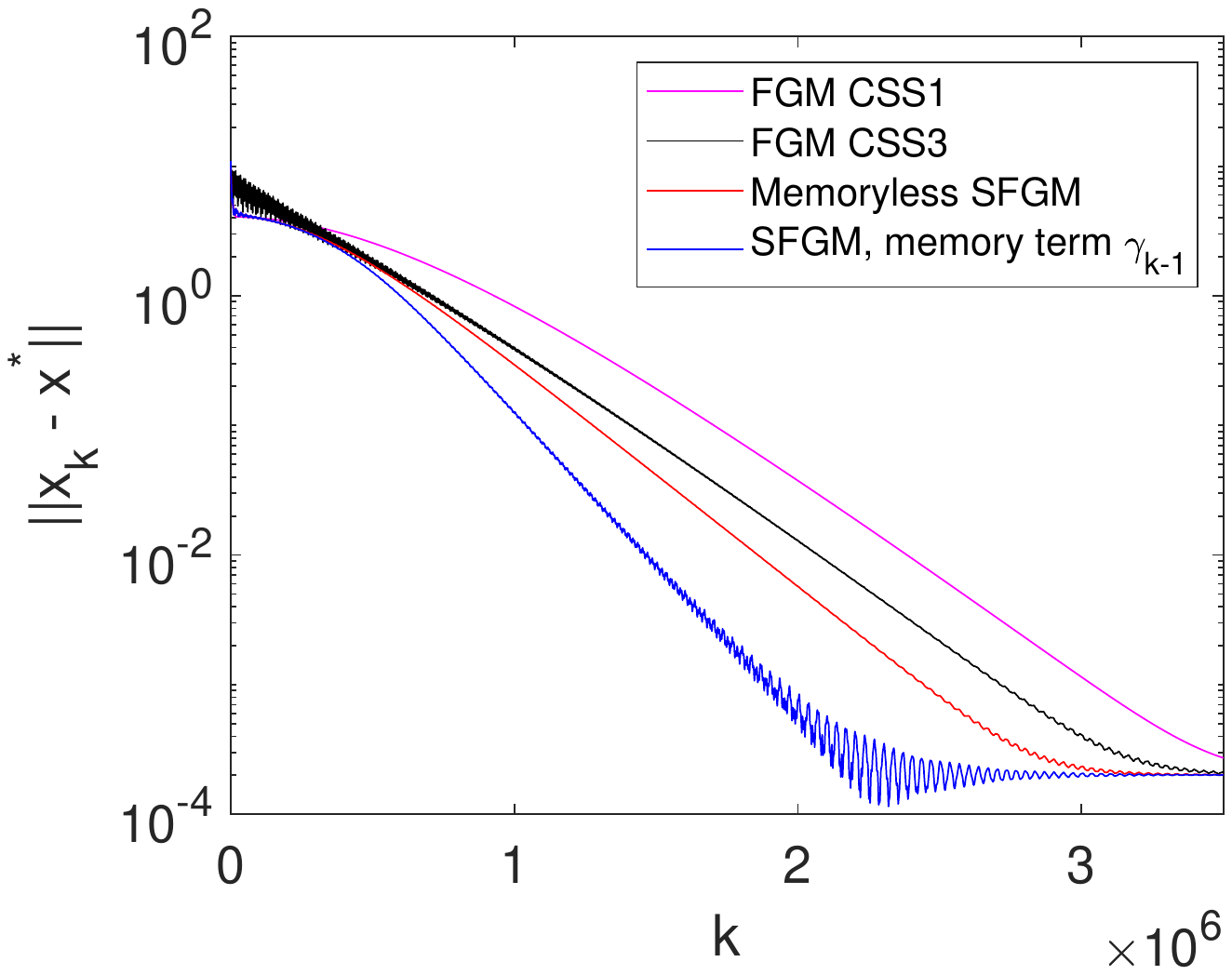} 		\label{num_sec_fig_3_sub_2} \vspace{-8mm}
			\caption{The dataset is ``triazines'' and $\tau = 10^{-8}$.}  \vspace{-8mm}
		\end{subfigure}
		\begin{subfigure}[h]{0.49\columnwidth}
			\centering
			\includegraphics[width = \columnwidth,height = \linewidth, trim={3cm 9.5cm 3cm 7.5cm}]{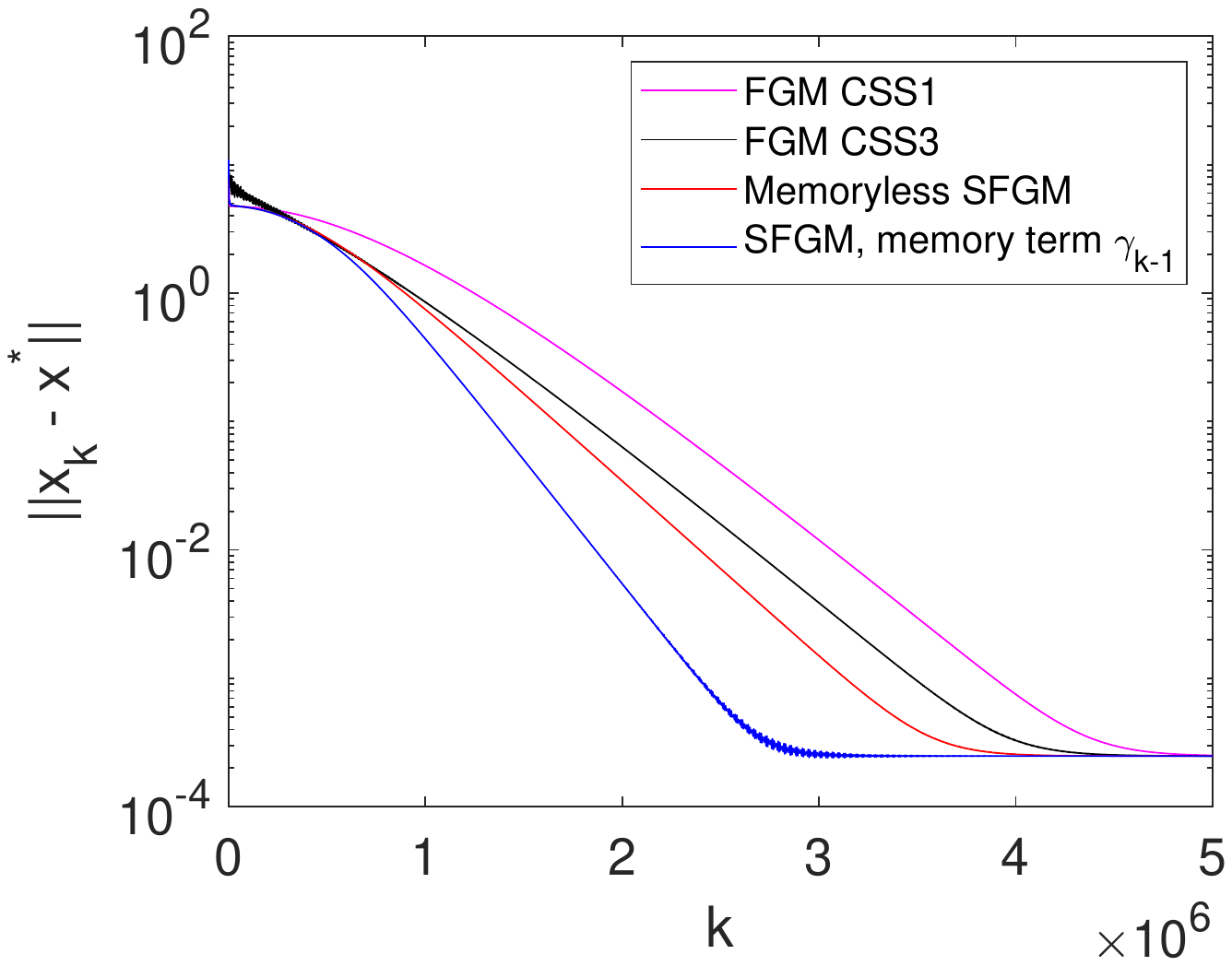} 		\label{num_sec_fig_3_sub_3}\vspace{-8mm}
			\caption{The dataset is ``a1a'' and $\tau = 10^{-7}$.}
		\end{subfigure}
		% \vfill
		\begin{subfigure}[h]{0.49\columnwidth}
			\centering
			\includegraphics[width=\columnwidth,height = \linewidth, trim={3cm 9.5cm 3cm 8cm}]{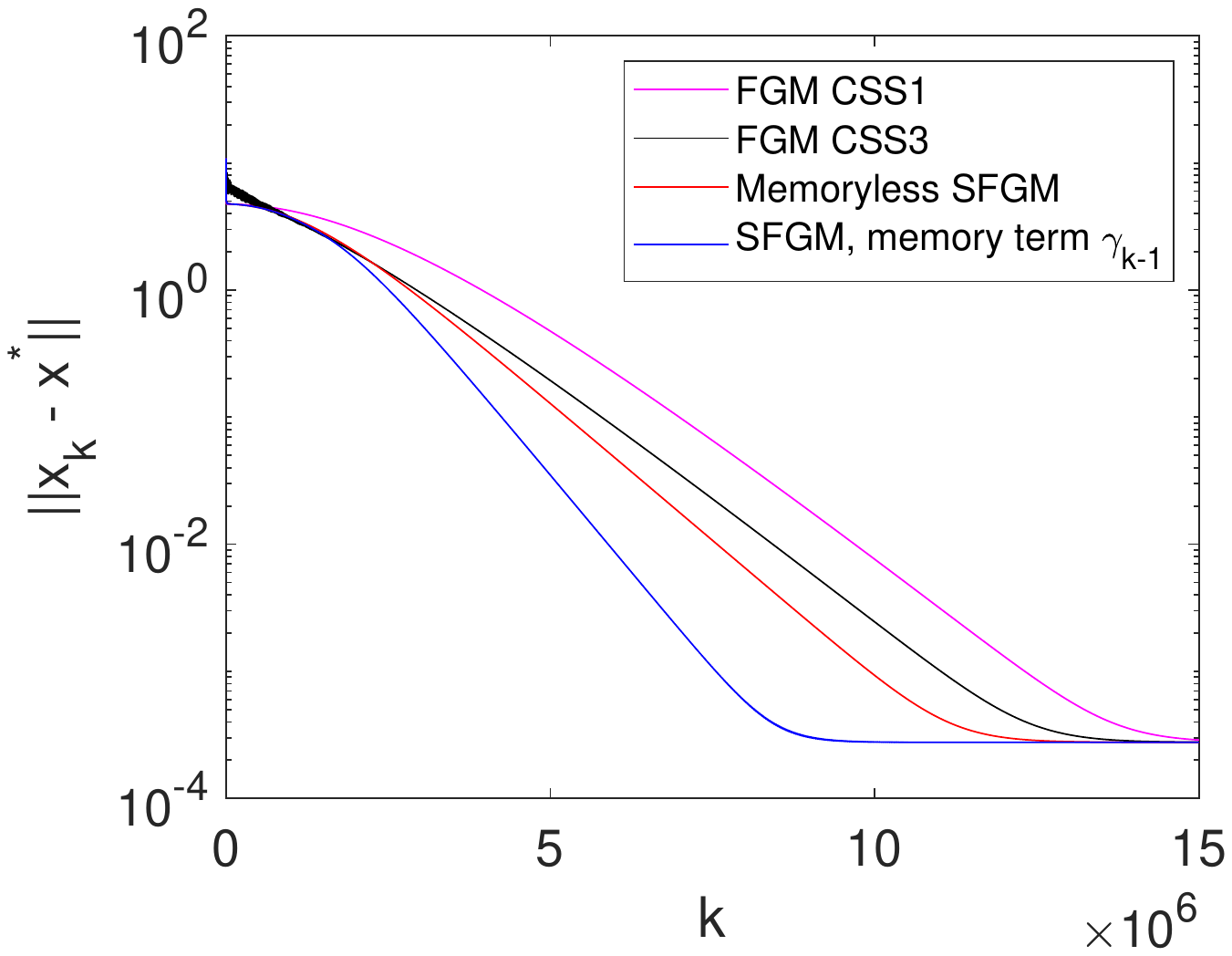} 		\label{num_sec_fig_3_sub_4}\vspace{-8mm}
			\caption{The dataset is ``a1a'' and $\tau = 10^{-8}$.} 
		\end{subfigure}
		\caption{Comparison between the efficiency of algorithms tested in minimizing the regularized quadratic loss function in the case where $m > n$, i.e., $A$ is a tall matrix.}
		\label{num_sec_fig_3}
	\end{figure}
	
	\vspace{-3mm}
	
	%%%%%%%%%%%%%%%%%%%%%%%%%%%%%%%%%%%%%%%
	\subsection{Decreasing the norm of the gradient}
	\label{Decreasing the norm of the gradient}
	\vspace{-1mm}
	%%%%%%%%%%%%%%%%%%%%%%%%%%%%%%%%%%%%%%%
	In many practical problems, it is of high interest to find points with small norm of the gradient, that is, 			\vspace{-3mm}
	\begin{equation}
		\label{grad____}
		||\nabla f(x)|| \leq \eta, 			\vspace{-3mm}
	\end{equation}
	where $\eta$ denotes the desired tolerance. In \cite{Nesterov_optima, Nesterov_book}, it is shown that FGM is not optimal in this sense. Instead, minimizing a regularized version of the objective function, which results in a reduction of the iteration complexity to $\mathcal{O} \sqrt{\frac{LR}{\epsilon}} \text{ln} \left(\frac{LR}{\epsilon}\right)$ is suggested therein. From this perspective, utilizing the construction of $\psi_k (x)$ proposed in \eqref{psii} in Definition \ref{def__1}, we can see that SFGM is minimizing a regularized version of the objective function. Moreover, when the generalized estimating sequences framework is used, it also provides the regularizer term, which consists of linear combinations of the previously constructed scanning functions weighted such that \eqref{psi_bound_} is satisfied. In the sequel, we show that the simplest versions of SFGM are more efficient than FGM in decreasing the norm of the gradient. 
	
	An example of the efficiency of SFGM in decreasing the norm of the gradient for minimizing the quadratic loss function has already been shown in Figs. \ref{num_sec_fig_1}(c) and \ref{num_sec_fig_1}(d). To diversify the nature of the problems solved, for these computational experiments, we consider the regularized logistic loss problem 
	\begin{equation}
		\begin{aligned}
			\label {num_eq_2}
			& \underset{x \in \mathcal{R}^n}{\text{minimize}}
			& &\frac{1}{m} \sum\limits_{i=1}^m \text{log} \left( 1 + \text{e}^{-b_i x a_i} \right) + \frac{\tau}{2} ||x||^2.  
		\end{aligned}
	\end{equation} 
	For this problem type, we reuse the datasets ``colon-cancer'' and ``a1a'', which were introduced in Section \ref{linear_regression}. We set $\tau \in \{10^{-5}, 10^{-7}\}$ for the ``colon-cancer'' dataset, and $\tau \in \{10^{-6}, 10^{-8}\}$ for the ``a1a'' dataset. The results are reported in Fig. \ref{num_sec_fig_4}. We can observe from Fig. \ref{num_sec_fig_4} that SFGM outperforms FGM for both datasets. Specifically, SFGM with memory term $\gamma_{k-1}$ is approximately $35 \% - 40 \%$ faster at decreasing the norm of the gradient than FGM CSS3.
	\begin{figure}
		\centering
		\begin{subfigure}[h]{0.49\columnwidth}
			\centering
			\includegraphics[width=1\columnwidth,height = \linewidth, trim={3cm 9.5cm 3cm 7.5cm}]{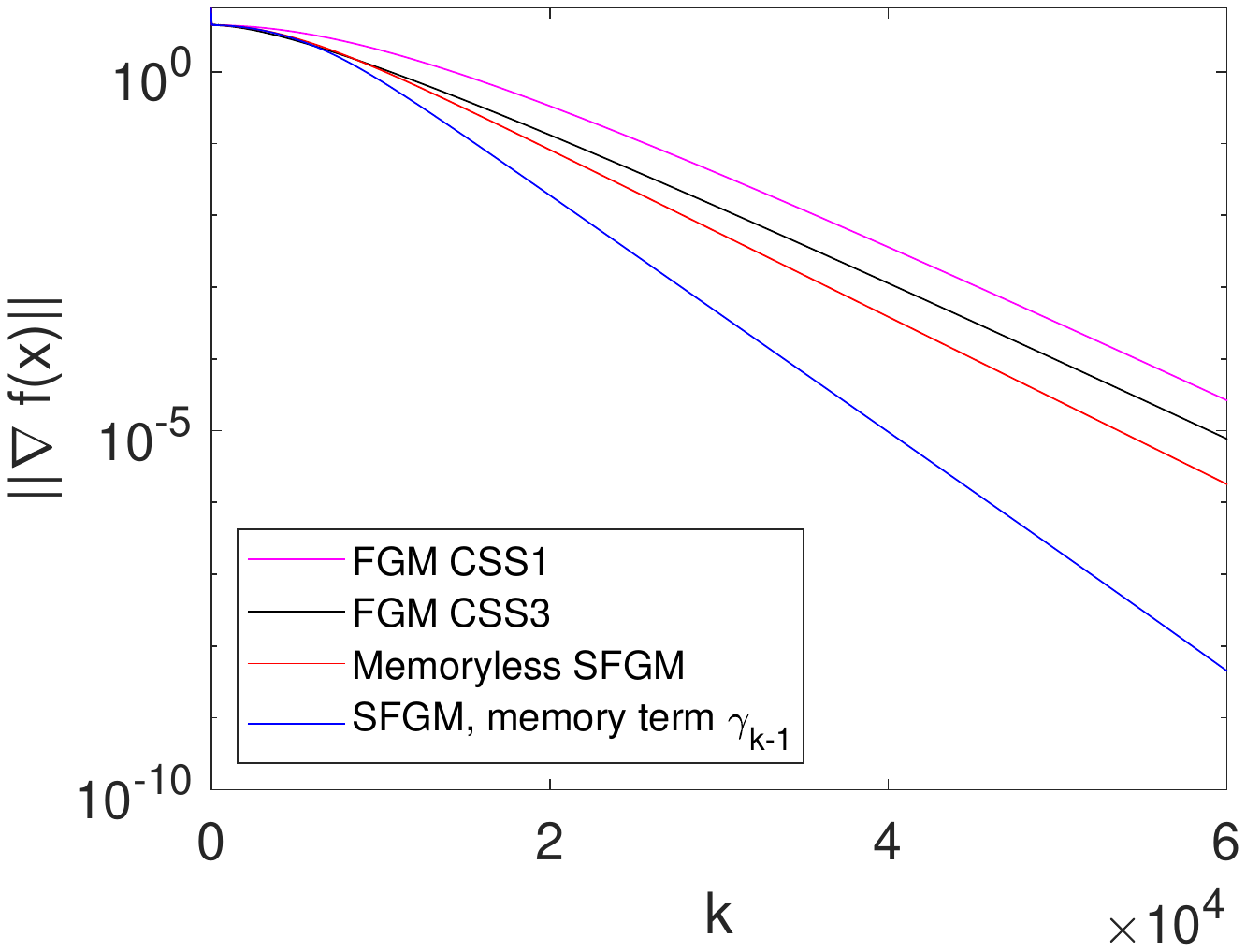} 		\label{num_sec_fig_4_sub_3}  \vspace{-8mm}
			\caption{The dataset is ``colon-cancer'' and $\tau = 10^{-5}$.} \vspace{-8mm}
		\end{subfigure}
		% \vfill
		\begin{subfigure}[h]{0.49\columnwidth}
			\centering
			\includegraphics[width=\columnwidth,height = \linewidth, trim={3cm 9.5cm 3cm 7.5cm}]{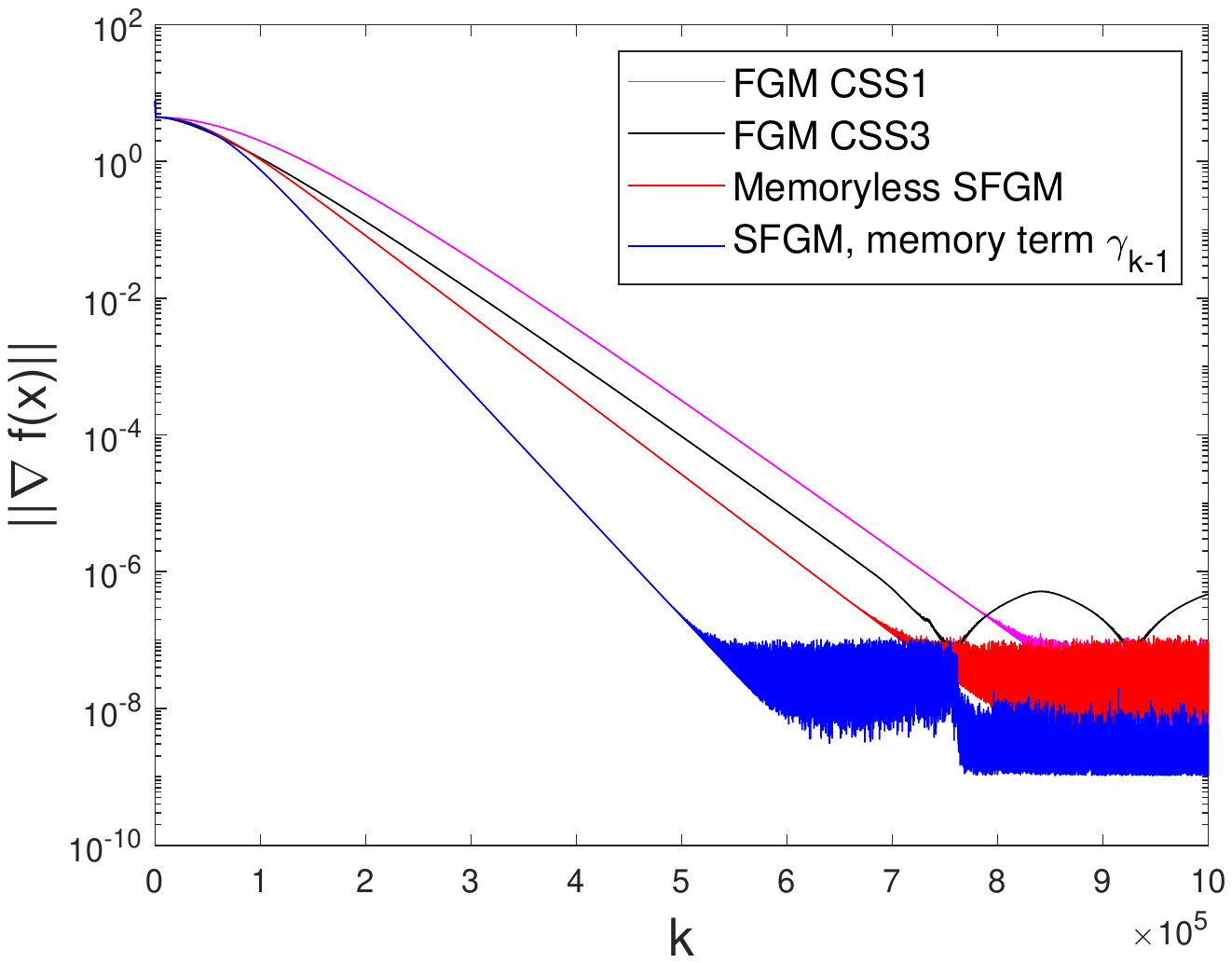} 		\label{num_sec_fig_4_sub_4} \vspace{-8mm}
			\caption{The dataset is ``colon-cancer'' and $\tau = 10^{-7}$.}  \vspace{-8mm}
		\end{subfigure}
		\begin{subfigure}[h]{0.49\columnwidth}
			\centering
			\includegraphics[width = \columnwidth,height = \linewidth, trim={3cm 9.5cm 3cm 7.5cm}]{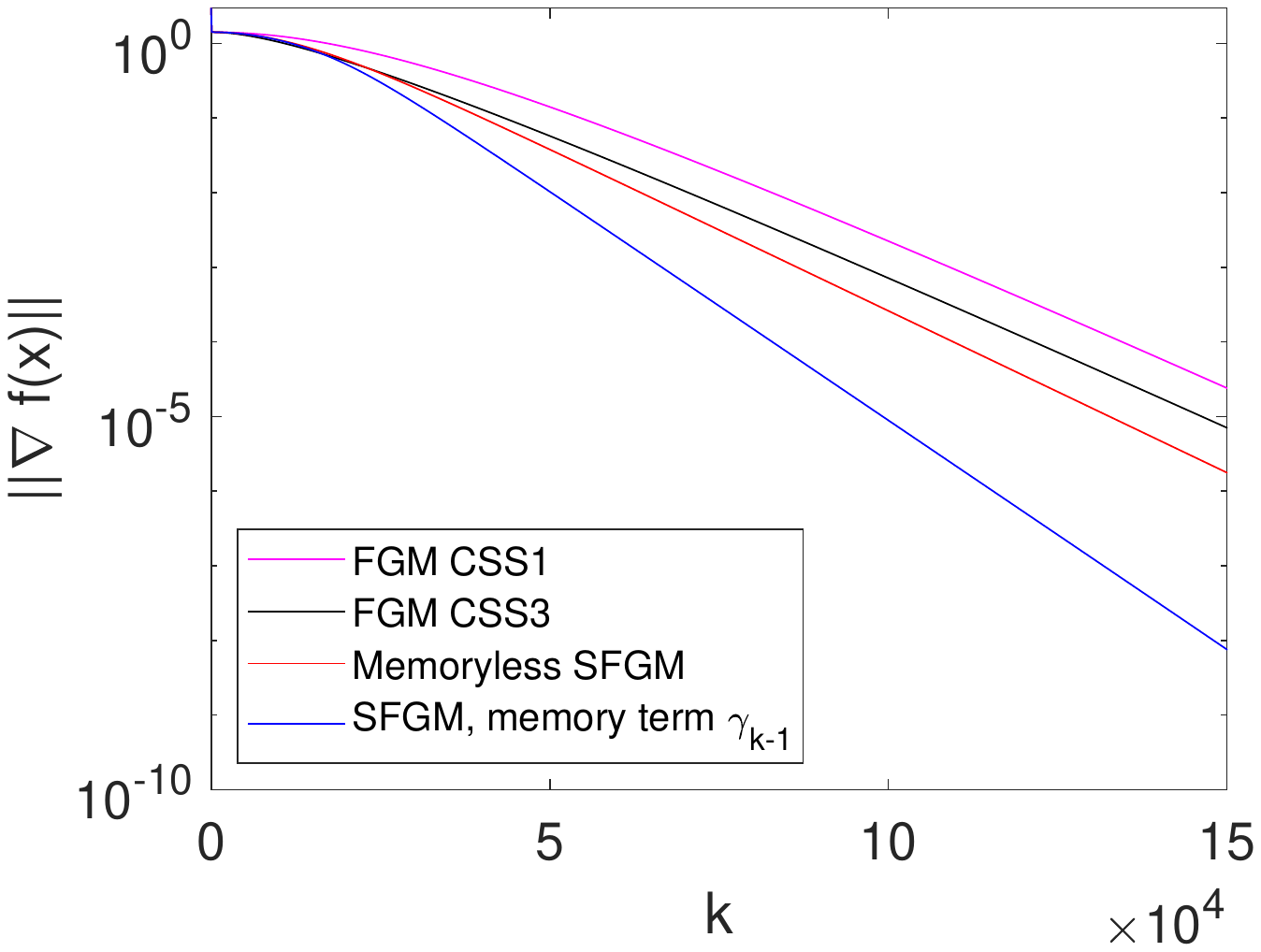} 		\label{num_sec_fig_4_sub_1} \vspace{-8mm}
			\caption{The dataset is ``a1a'' and $\tau = 10^{-6}$.} 
		\end{subfigure}
		% \vfill
		\begin{subfigure}[h]{0.49\columnwidth}
			\centering
			\includegraphics[width=\columnwidth,height = \linewidth, trim={3cm 9.5cm 3cm 7.5cm}]{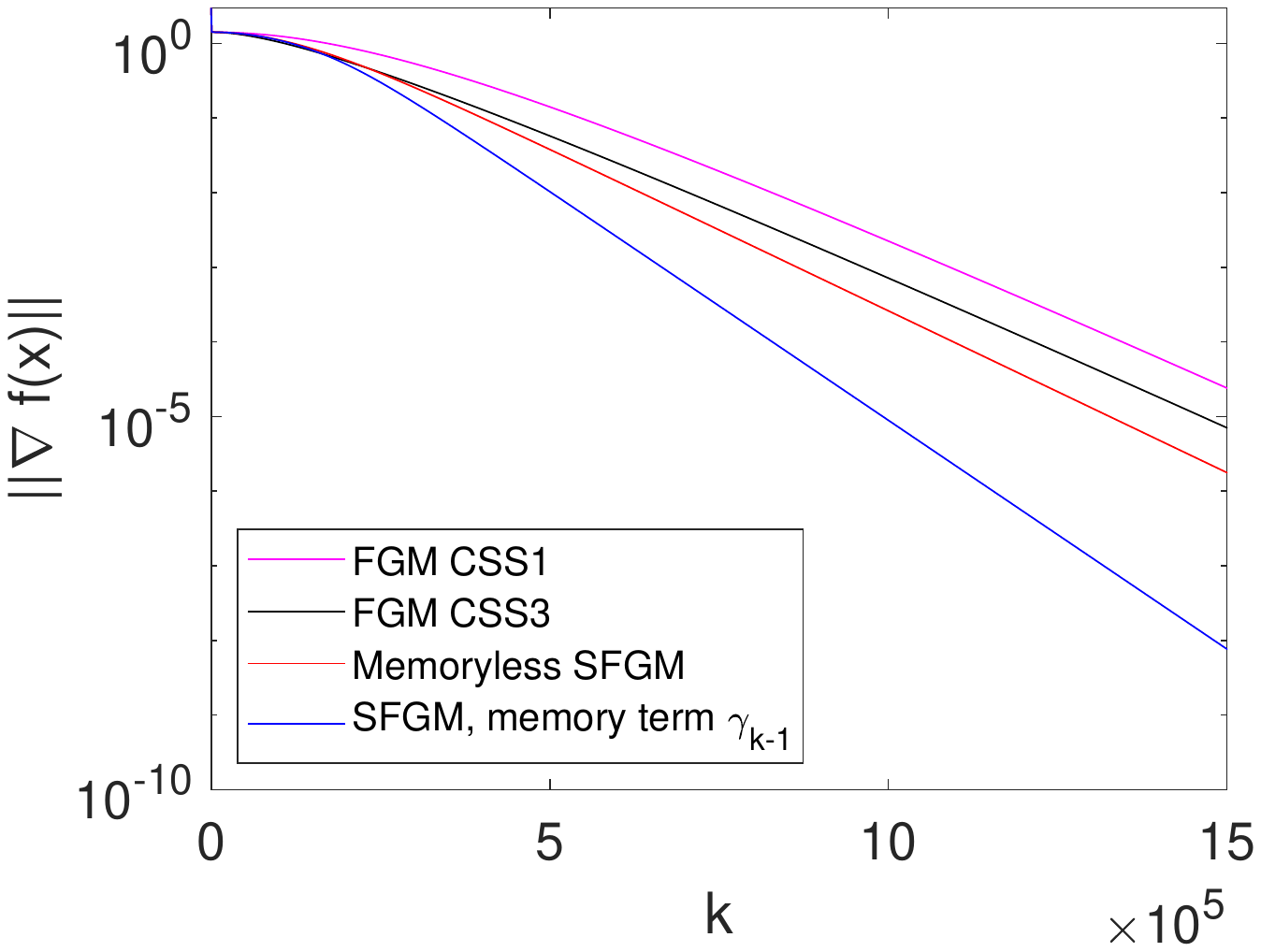} 		\label{num_sec_fig_4_sub_2} \vspace{-8mm}
			\caption{The dataset is ``a1a'' and $\tau = 10^{-8}$.} 
		\end{subfigure}
		\caption{Comparison between the efficiency of algorithms tested in minimizing the regularized logistic loss function for various datasets.}
		\label{num_sec_fig_4}
	\end{figure}

	\vspace{-4mm}
	
	%%%%%%%%%%%%%%
	\section{Conclusion and Discussion}
	\label{Discussion}
	%%%%%%%%%%%%%%
	{The way for embedding a new form of heavy-ball momentum into Nesterov's acceleration framework has been rigorously established, and shown to be of practical significance for solving typical signal processing problems.} %{It is also rigorously proven that utilizing the two acceleration principles results in schemes that minimize regularized versions of the objective function.} 
	The faster convergence (than FGM) of the proposed accelerated algorithm that we name SFGM is established analytically and demonstrated through simulations and real data analysis. One more novelty important for this venue is that we also provide intuition on the design of accelerated methods based on the example of the proposed SFGM, which was, in fact, necessary for our objective of deriving new methods that result from the embedding/coupling of different acceleration principles in one scheme.

	We conclude this work by introducing several open problems that arise from our proposed framework. Several interesting questions that arise are related to the design of the sequence $\{\psi_k(x)\}_{k=0}^\infty$. Considering the construction of $\psi_k(x)$ in \eqref{psii}, the most interesting issue is finding the optimal (in some sense) selection of the coefficients $\beta_{i,k}$. This would produce the optimal regularizers for the objective function, which would result in faster algorithms. These optimal regularizers can be utilized to construct methods that are optimal in the sense of decreasing the norm of the gradient (see also the discussion in Section~\ref{Decreasing the norm of the gradient}). The latter topic has gathered significant attention in the recent years, particularly with the advances in nonconvex optimization \cite{ghadimi2016accelerated, Lower_Bounds, liang2019average}, wherein the goal is to find stationary points of the function that satisfy \eqref{grad____}. 
	
	Another topic of interest is related to devising alternative candidate structures for the term $\psi_k (x)$, which can ideally encompass both black and white box information about the objective function. This idea is inspired by the work in \cite{doi:10.1137/16M1099546}, wherein the authors develop the notions of relative smoothness and relative strong convexity. %They provide examples of differentiable convex functions, which do not satisfy the uniform smoothness condition, i.e. Lipschitz constant with finite value, which can be smooth and strongly convex relative to a simpler function $h(x)$. 
	For twice differentiable functions, the relative smoothness and strong convexity parameters depend on the weighted difference of the Hessians of the cost function and $h(x)$ \cite[Proposition 1.1]{doi:10.1137/16M1099546}. Apparently, a similar approach was also developed here in establishing \eqref{psi_bound_}, with the main difference being that in our case $\psi_k(x)$ is dynamically changing over iterations. From the perspective of the framework introduced in \cite{doi:10.1137/16M1099546}, the result obtained in \eqref{psi_bound_} suggests that the relative strong convexity parameter between $f(x)$ and $\psi_k(x)$ is not unique. Instead, it is contained in an interval which shrinks over iterations, and as $k \rightarrow \infty$, it is contained in $[0,1]$. Thus, it is of interest to study how the two frameworks can be coupled. 
	
	Another strategy that is known to improve the performance of FGM is restarting. %the method. 
	Several restarting conditions have been introduced \cite{candes_donoghue, 7040179}. These conditions can also be applied to SFGM and improve its performance as well. In this work, we purposely avoided relying on heuristics like restarting for further improving the performance of SFGM. Nevertheless, it is of interest to establish restarting conditions applicable to SFGM. Lastly, it would be of interest to investigate extensions of the proposed framework to solve nonsmooth optimization problems. To solve such problems, several variations of FGM already exist \cite{Nesterov_2007, FISTA, Tseng}. 
	
	\vspace{-4mm}
	
	%%%%%%%%%%%%%%
	\section*{Acknowledgments}
	\vspace{-1mm}
	%%%%%%%%%%%%%%
	We would like to thank Professor Yurii Nesterov for his suggestions and fruitful discussions about the early version of the article.
	\vspace{-5mm}
	
	\appendices
	\section{Proof of Lemma 1}
	\label{Proof of Lemma 1}
	%%%%%%%%%%%%%%%%%%%%%%%%%%%%%%%%%%%%%%%%%%%%%%%%%%%%%%%%%%%%
	By assumption that is made in the formulation of the lemma, we can write
	\begin{align}
		f(x_k) \leq \Phi_{k}^* &= \underset{x \in {\mathcal{R}^n}}{\text{min}} \Phi_{k} (x) \\ \nonumber &\stackrel{\eqref{def_1}}{\leq} 
		\underset{x \in {\mathcal{R}^n}}{\text{min}} \left[\lambda_{k} \Phi_{0}(x^*) + (1 - \lambda_{k}) \left(f(x) - \psi_k(x)\right) \right] \\ \nonumber
		&\leq \left[ \lambda_{k} \Phi_{0}(x^*) + (1 - \lambda_{k}) \left(f(x^*) - \psi_k(x^*)\right) \right].
	\end{align}
	Rearranging the terms yields the desired result. 
	
	\section{Proof of Lemma 2} 
	\label{Proof of Lemma 2}
	%%%%%%%%%%%%%%%%%%%%%%%%%%%%%%%%%%%%%%%%%%%%%%%%%%%%%%%%%%%%
	We prove the lemma by induction. At iteration $k = 0$, since $\psi_0(x) = 0$, $\Psi_0 = 0$ and $\lambda_0 = 1$, utilizing \eqref{def_1} in the paper, we have $\Phi_0(x) \leq \lambda_0 \Phi_0(x) + \left( 1 - \lambda_0 \right) f(x) \equiv \Phi_0 (x)$. Next, we assume that at some iteration $k$, \eqref{def_1} holds true, which yields
	\begin{align}
		\label{useful}
		\Phi_k(x) - \left(1 - \lambda_k \right) f(x) \leq \lambda_k \Phi_0 (x) - \left(1 - \lambda_k \right) \psi_k (x).
	\end{align}
	
	Utilizing \eqref{lower_bound} and \eqref{phi_k+1_SFGM}, at iteration $k + 1$ we can write  
	\begin{align}
		\label{crappy}
		\Phi_{k+1} (x) &\leq (1 - \alpha_k) \left(\Phi_{k} (x) + \psi_k(x) \right) - \psi_{k+1} (x) - \Psi_k  + \alpha_k \left( f(x) + \psi_k (x) \right).
	\end{align}
	Then, utilizing \eqref{Psi_definition}, adding and subtracting the same term to the right hand side (RHS) of \eqref{crappy}, we have 
	\begin{align}
		\Phi_{k+1} (x) &\leq (1 \! - \! \alpha_k) \Phi_{k} (x) \! - \! \psi_{k+1} (x) \! + \! \alpha_k f(x) + (1 \! - \! \alpha_{k}) (1 \! - \! \lambda_{k}) f(x) \! - \! (1 \! - \! \alpha_{k}) (1 \! - \! \lambda_{k}) f(x) \\ \label{3.8}
		&= (1 - \alpha_k) \left[\Phi_{k} (x) - (1 - \lambda_{k}) f(x)\right] - \psi_{k+1}(x) + \left(\alpha_k + (1 - \lambda_{k}) (1 - \alpha_k) \right) f(x).
	\end{align}
	Utilizing \eqref{useful} in \eqref{3.8}, we obtain
	\begin{align}
		\label{mmmm}
		\Phi_{k+1} (x) \! + \! \psi_{k+1} (x) \! &\leq \! (1 \! - \! \alpha_k) \! \left(\lambda_{k} \Phi_{0} (x) - \left(1 \! - \! \lambda_k \right) \psi_k (x) \right) + (1 - \lambda_{k} + \alpha_{k} \lambda_{k}) f(x).
	\end{align}
	Then, from the recursive relation \eqref{lambda_recursive}, and also by relaxing the RHS of \eqref{mmmm}, we reach
	\begin{align}
		\label{mm}
		\Phi_{k+1} (x) + \psi_{k+1} (x) \leq \lambda_{k+1} \Phi_{0} (x)  + (1 - \lambda_{k+1})f(x).
	\end{align}
	Finally, utilizing the fact that $\lambda_k \in [0,1]$, we obtain
	\begin{align}
		\label{m}
		\Phi_{k+1} (x) \leq \lambda_{k+1} \Phi_{0} (x)  + (1 - \lambda_{k+1}) \left(f(x) - \psi_{k+1}(x) \right).
	\end{align}

	%%%%%%%%%%%%%%%%%%%%%%%%%%%%%%%%%%%%%%%%%%%%%%%%%%%%%%%%%%%%
	\section{Proof of Lemma 3}
	\label{Proof of Lemma 3}
	%%%%%%%%%%%%%%%%%%%%%%%%%%%%%%%%%%%%%%%%%%%%%%%%%%%%%%%%%%%%
	Let us begin by establishing that \eqref{phi_k+1_SFGM} preserves the quadratic structure of the terms in the sequence $\{\Phi_k\}_{k=0}^\infty$. Note that at step $k = 0$, we have $\psi_0 = 0$. Therefore, $\nabla^2 \Phi_{0} (x) = \nabla^2 \phi_{0} (x) = \gamma_0 I$. Next, let us assume that for some step $k$, we have $\nabla^2 \Phi_k(x) = \gamma_k - \sum_{i = 0}^{k - 1} \beta_{i,k} \gamma_{i} \stackrel{\eqref{psi_bound_}}{\geq} 0$. Then, by considering the Hessian of \eqref{phi_k+1_SFGM}, we can write 
	\begin{align}
		\label{tt}
		\nabla^2 \Phi_{k+1} (x) \! &\stackrel{}{=} \! (1 \!- \! \alpha_k) \gamma_{k}  I \!- \! \sum_{i = 0}^{k} \beta_{i,k} \gamma_i I \!  + \! \alpha_k \! \left( \! \! \mu \! + \! \sum_{i = 0}^{k - 1} \! \beta_{i,k} \gamma_{i} \! \! \right) \! \! I \! . 
	\end{align}
	Utilizing \eqref{gamma_expr} in \eqref{tt} we obtain 
	\begin{align}
		\label{ttt}
		\nabla^2 \Phi_{k+1} (x) &\stackrel{}{=} \gamma_{k+1} I - \sum_{i = 0}^{k} \beta_{i,k} \gamma_i I. 
	\end{align}
	%\ED{Does it cause problems with the inductive argument that we prove from indexes k and k-1 (from assumption at step k), to indexes k+1 and k. }
	Lastly, we note that selecting the terms $\beta_{i,k}$ to satisfy \eqref{psi_bound_} ensures that $\nabla^2 \Phi_{k+1} (x) \geq 0$. 
	
	We proceed now to establishing the recursive relation for the terms in the sequence $\{v_k\}_{k=0}^\infty$. Let us start by substituting our proposed construction for the scanning function presented in \eqref{scan_funct_k_SFGM} into \eqref{phi_k+1_SFGM}, and making the necessary manipulations we obtain
	\begin{align}
		\label{FDD}
		\phi_{k+1}^* + \frac{\gamma_{k+1}}{2}||x - v_{k+1}||^2 &= \! (1 \! - \! \alpha_k) \! \left( \! \phi_{k}^* + \frac{\gamma_{k}}{2}||x - v_{k}||^2 \! \right) \! - \! \Psi_k \! \\ \nonumber &+ \! \alpha_{k} \! \left( \! f(y_k) + \nabla f(y_k)^T (x - y_k)  + \frac{\mu}{2} ||x - y_k||^2 + \psi_{k} (x) \right).
	\end{align}
	First, observe that both the LHS and the RHS of \eqref{FDD} are convex functions in $x$, and minimizing them over all possible values of $x$ yields two unconstrained optimization problems. Therefore, the solution needs to satisfy the optimality condition for unconstrained problems, which is that the gradient of the objective function with respect to the optimization parameter has to be equal to $0$. Taking gradients with respect to $x$, we have
	\begin{align}
		\label{SFGM_opt_cond}
		\begin{split}
			\gamma_{k+1}(x \! - \! v_{k+1}) \! &= \gamma_k (1-\alpha_k)(x - v_{k}) + \alpha_k \left(\mu (x - y_k) + \nabla f(y_k) + \sum_{i = 0}^{k - 1} \beta_{i,k} \gamma_{i} (x - v_{i}) \right).
		\end{split}
	\end{align}
	For now, assume that the points $y_k$ are known and the points $x$ are unknown. By utilizing \eqref{gamma_expr}, we can reduce the unknown points $x$ in \eqref{SFGM_opt_cond}. Then, after making some manipulations we obtain 
	\begin{align}
		\label{vvv}
		\begin{split}
			v_{k+1}  &= \frac{1}{\gamma_{k+1}}\left((1-\alpha_k)\gamma_k v_{k} + \mu \alpha_k \left(y_k - \frac{1}{\mu}\nabla f(y_k) + \sum_{i = 0}^{k - 1} \frac{\beta_{i,k} \gamma_{i}}{\mu} v_{i}\right)\right).
		\end{split}
	\end{align}
	
	Notice that the sequence $\{v_k\}_{k=0}^\infty$ depends on the sequence $\{y_k\}_{k=0}^\infty$, which is assumed to be known up to this point. We will show later how this value can be computed recursively. For now, let us focus on finding the smallest value of the scanning function, $\phi_{k+1}^*$. On a conceptual level, the simplest way to compute $\phi_{k+1}^*$ is to think that there is another scanning function $\Theta_k(y_k)$ for the sequence $\{y_k\}_{k=0}^\infty$, which has the same center and radius and  as the sequence of functions $\{\Phi_{k}(x)\}_{k=0}^\infty$. So, we have
	\begin{equation}
		\Theta_{k}(y_k) = \theta_{k}^* + \frac{\gamma_{k}}{2}||y_k - v_{k}||^2-  \sum_{i = 0}^{k-1} \beta_{i,k} \frac{\gamma_{i}}{2} ||y_k - v_{i}||^2, \; \forall k. \label{y_seq_scan_funct_k_SFGM}
	\end{equation}
	Then, utilizing \eqref{phi_k+1_SFGM} applied at the points $x = y_k$, we have
	\begin{align}
		\label{psi_k+1 bound}
		\Theta_{k+1} (y_k) &= (1 - \alpha_k) \left(\Theta_{k} (y_k) + \psi_k (y_k) \right) - \psi_{k+1} (y_k) - \Psi_k + \alpha_{k} \left( f(y_k) + \psi_{k} (y_k) \right).
	\end{align}
	Substituting \eqref{psii} and \eqref{y_seq_scan_funct_k_SFGM} into \eqref{psi_k+1 bound}, as well as making the necessary relaxations, we obtain
	\begin{align}
		\label{gen_conditions_thetas}
		\theta_{k+1}^* \! \! + \! \frac{\gamma_{k+1}}{2}||y_k \! - \! v_{k+1}||^2  \! &\leq (1 \! - \! \alpha_k) \! \left(\theta_k^*  + \frac{\gamma_k}{2}||y_k - v_k||^2 \right) + \! \alpha_k \! \! \left(\! \! f(y_k) \! + \! \sum_{i = 0}^{k-1} \! \frac{\beta_{i,k} \gamma_{i}}{2} \! ||y_k \! - \! v_{i}||^2 \! \right) \! \! .
	\end{align}

	From the recursive relation \eqref{vvv}, we have
	\begin{align}
		\label{gg}
		\begin{split}
			v_{k+1} - y_k  &= \frac{1}{\gamma_{k+1}}\left((1-\alpha_k)\gamma_k v_{k} + \mu \alpha_k \left(y_k - \frac{1}{\mu}\nabla f(y_k) + \sum_{i = 0}^{k - 1} \frac{\beta_{i,k} \gamma_{i}}{\mu} v_{i}\right) - \gamma_{k+1} y_k \right).
		\end{split}
	\end{align}
	Then, substituting the recursive relation for the term $\gamma_{k+1}$, i.e., \eqref{gamma_expr} into \eqref{gg}, yields
	\begin{align}
		\label{ggg}
		\begin{split}
			v_{k+1} - y_k  &= \frac{1}{\gamma_{k+1}}\left((1-\alpha_k)\gamma_k \left(v_{k} - y_k \right) - \alpha_k \nabla f(y_k) + \alpha_k \sum_{i = 0}^{k - 1} \frac{\beta_{i,k} \gamma_{i}}{\mu} \left(v_{i} - y_k \right) \right).
		\end{split}
	\end{align}
	Taking $||\cdot||^2$ of both sides in \eqref{ggg}, we obtain 
	\begin{align}
		\label{more_eq}
		||v_{k+1} - y_k||^2 = \frac{||\left(\gamma_k(1-\alpha_k)(v_{k} - y_k )\right) \! + \! \alpha_k \! \sum_{i = 0}^{k - 1} \beta_{i,k} \gamma_{i} (v_{i} - y_k) - \alpha_k \nabla f(y_k) ||^2}{\gamma_{k+1}^2}.
	\end{align} 
	Then, multiplying both sides of \eqref{more_eq} by $\frac{\gamma_{k+1}}{2}$, and expanding the RHS, we obtain 
	\begin{align}
		\label{fgm_eq_1}
		\frac{\gamma_{k+1}}{2}||v_{k+1} - y_k||^2 &= \frac{(1-\alpha_k)^2\gamma_k^2}{2 \gamma_{k+1}} ||v_{k} - y_k||^2 + \frac{\alpha_k^2}{2 \gamma_{k+1}} ||\sum_{i = 0}^{k - 1} \! \beta_{i,k} \gamma_{i} (v_{i} - y_k) ||^2 \\ \nonumber &- \frac{2 \alpha_k (1-\alpha_k)\gamma_{k}}{2 \gamma_{k+1}} (v_{k} - y_k)^T\nabla f(y_k)   + \! \frac{\alpha_{k}^2}{2 \gamma_{k+1}} ||\nabla f(y_k) ||^2 \!  \\ \nonumber &+ \! \frac{(1 \! - \! \alpha_k) \alpha_{k} \gamma_{k}}{\gamma_{k+1}} \! \sum_{i = 0}^{k - 1} \! \beta_{i,k} \gamma_{i} (v_{i} \! - \! y_k)^T \! (v_{k} \! - \! y_k) \! - \! \frac{\alpha_k^2}{\gamma_{k+1}} \! \sum_{i = 0}^{k - 1} \! \beta_{i,k} \gamma_{i} (y_k - v_{i})^T  \! \nabla f(y_k)  \! . 		
	\end{align} 
	
	Substituting \eqref{fgm_eq_1} into \eqref{gen_conditions_thetas} and doing the respective factorings, we obtain 
	\begin{equation} 
		\label{fgm_eq_3}
		\begin{split}
			\theta_{k+1}^* &\leq \alpha_k f(y_k) + (1-\alpha_k) \theta_{k}^* + \frac{(1-\alpha_{k}) \gamma_{k}}{2}\left[\frac{\gamma_{k+1}}{\gamma_{k+1}} - \frac{(1 - \alpha_{k}) \gamma_{k}}{\gamma_{k+1}}\right] ||y_k - v_{k}||^2 \\ &- \frac{\alpha_{k}^2}{2 \gamma_{k+1}} ||\sum_{i = 0}^{k-1} \frac{\beta_{i,k} \gamma_{i}}{2}(y_k - v_{i}) ||^2 - \frac{\alpha_k^2}{2 \gamma_{k+1}} ||\nabla f(y_k) ||^2 + \frac{\alpha_k^2}{\gamma_{k+1}} \sum_{i = 0}^{k - 1} \beta_{i,k} \gamma_{i} (v_{i} - y_k) ^T \nabla f(y_k) \\ &+ \frac{\alpha_k (1 \! - \! \alpha_k)\gamma_k}{\gamma_{k+1}} \left( \! (v_{k} \! - \! y_k)^T \! \nabla f(y_k) \! - \! \! \sum_{i = 0}^{k - 1} \! \beta_{i,k} \gamma_{i} (y_k \! - \! v_{i})^T \! (y_k \! - \! v_{k})\right) \! + \! \alpha_{k} \! \sum_{i = 0}^{k-1} \! \frac{\beta_{i,k} \gamma_{i}}{2} ||y_k - v_{i} ||^2 \! \!.
		\end{split}
	\end{equation}
	Making some further manipulations and relaxing the upper bound on $\theta_{k+1}^*$ in \eqref{fgm_eq_3} yields 
	\begin{equation}
		\label{fgm_eq_4}
		\begin{split}
			\theta_{k+1}^* \! &\leq \! \alpha_k f(y_k) \! + \! (1 \! - \! \alpha_k) \theta_{k}^* \! + \! \frac{\alpha_{k} \gamma_{k}(1-\alpha_{k}) (\mu + \sum_{i = 1}^{k-1} \beta_{i,k} \gamma_{i})}{2\gamma_{k+1}} ||y_k - v_{k}||^2 \! \\ &+ \! \alpha_{k} \! \sum_{i = 0}^{k-1} \! \frac{\beta_{i,k} \gamma_{i}}{2} ||y_k - v_{i} ||^2 - \frac{\alpha_k^2}{2 \gamma_{k+1}} ||\nabla f(y_k) ||^2 \\ &+ \frac{\left(1 - \alpha_{k}\right)\alpha_k^2}{\gamma_{k+1}} \sum_{i = 0}^{k - 1} \beta_{i,k} \gamma_{i} (v_{i} - y_k) ^T \nabla f(y_k) + \frac{\alpha_k^3}{\gamma_{k+1}} \sum_{i = 0}^{k - 1} \beta_{i,k} \gamma_{i} (v_{i} - y_k) ^T \nabla f(y_k) \\ &+ \frac{\alpha_k (1-\alpha_k)\gamma_k}{\gamma_{k+1}} \left( (v_{k} - y_k)^T\nabla f(y_k) - \sum_{i = 0}^{k - 1} \beta_{i,k} \gamma_{i} (y_k - v_{i})^T \left(y_k - v_{k}\right) \right).
		\end{split}
	\end{equation}
	Then, utilizing the Cauchy-Schwartz inequality in \eqref{fgm_eq_4}, as well as relaxing the upper bound, we obtain 
	\begin{equation}
		\label{fgm_eq_5}
		\begin{split}
			\theta_{k+1}^* &\stackrel{}{\leq} \alpha_k f(y_k) + (1-\alpha_k) \theta_{k}^* + \frac{\alpha_{k} \gamma_{k}(1-\alpha_{k}) (\mu + \sum_{i = 1}^{k-1} \beta_{i,k} \gamma_{i})}{2\gamma_{k+1}} ||y_k - v_{k}||^2 \\ &+ \alpha_{k} \sum_{i = 0}^{k-1} \frac{\beta_{i,k} \gamma_{i}}{2} ||y_k - v_{i} ||^2 - \frac{\alpha_k^2}{2 \gamma_{k+1}} ||\nabla f(y_k) ||^2 + (1-\alpha_k)\frac{\gamma_{k}}{2}||x_{\Phi_k}^* \! - \! v_{k}||^2  \\ &+ \frac{\left(1 - \alpha_{k}\right)\alpha_k^2}{\gamma_{k+1}} \sum_{i = 0}^{k - 1} \beta_{i,k} \gamma_{i} (v_{i} - y_k) ^T \nabla f(y_k) + \frac{\alpha_k^3}{\gamma_{k+1}} \sum_{i = 0}^{k - 1} \beta_{i,k} \gamma_{i} ||v_{i} - y_k||\;  ||\nabla f(y_k)||  \\ &+ \frac{\alpha_k (1 \! - \! \alpha_k)\gamma_k}{\gamma_{k+1}} \left( \! \left(v_{k} \! - \! y_k\right)^T \! \nabla f(y_k) \! + \! \sum_{i = 0}^{k - 1} \! \beta_{i,k} \gamma_{i} ||y_k \! - \! v_{i}|| \; ||y_k \! - \! v_{k}||\!\right) \! + \! \sum_{i = 0}^{k-1} \! \beta_{i,k} \frac{\gamma_{i}}{2} ||x_{\Phi_k}^* \! - \! v_{i}||^2 \!\!. 
		\end{split}
	\end{equation}
	Lastly, since we would like the scanning function to be as close as possible to the objective function itself, we let $\theta_{k+1}$ equal to the tightest upper bound we can obtain analytically. Moreover, as discussed earlier, we let $\phi_{k}^* = \theta_k^*, \; \forall k = 0, 1, \ldots$. This way we obtain \eqref{psi_{k+1}^*}.
	
	%%%%%%%%%%%%%%%%%%%%%%%%%%%%%%%%%%%%%%%%%%%%%%%%%%%%%%%%%%%%
	\section{Proof of Theorem 1}
	\vspace{-1mm}
	\label{Proof of Theorem 1}
	%%%%%%%%%%%%%%%%%%%%%%%%%%%%%%%%%%%%%%%%%%%%%%%%%%%%%%%%%%%%
	Let $\phi_0^* = f(x_0)$. Then, by construction of the scanning function at iteration $k = 0$, we have $f(x_0) \leq \Phi_0(x) = f(x_0) + \frac{\gamma_0}{2} ||x - x_0||^2$. Moreover, we recall that the update rules of SFGM were devised to maintain the relation $f(x_k) \leq \Phi_k^*$. This is sufficient for the results proved in Lemma~\ref{SFGM_lemma_1} to be applied. 
	
	\vspace{-5mm}
	
	%%%%%%%%%%%%%%%%%%%%%%%%%%%%%%%%%%%%%%%%%%%%%%%%%%%%%%%%%%%%
	\section{Proof of Lemma 4}
	\label{Proof of Lemma 4}
	%%%%%%%%%%%%%%%%%%%%%%%%%%%%%%%%%%%%%%%%%%%%%%%%%%%%%%%%%%%%
	From \eqref{gamma_expr}, we can write
	\begin{align}
		\nonumber
		\gamma_{k+1} \! - \! \left( \! \mu \! + \! \sum_{i = 1}^{k-1} \! \beta_{i,k} \gamma_i \! \right) \! &= \! (1 \! - \! \alpha_k) \gamma_{k} \! + \! \alpha_k \! \left( \! \mu \! + \! \sum_{i = 1}^{k-1} \! \beta_{i,k} \gamma_i \! \right) - \left(\mu + \sum_{i = 1}^{k-1} \beta_{i,k} \gamma_i\right) \\ \label{kkkk}
		&= \! (1 \! - \! \alpha_k) \lambda_0 \! \! \left[ \! \gamma_{k} \! - \! \left( \! \mu \! + \! \sum_{i = 1}^{k-1} \! \beta_{i,k} \gamma_i \! \right) \! \right] \! .
	\end{align}
	Then, utilizing the recursivity of \eqref{gamma_expr} in \eqref{kkkk}, we obtain
	\begin{align}
		\label{FGM_conv_eq_1}
		\gamma_{k+1} \! - \! \left( \! \mu \! + \! \sum_{i = 1}^{k-1} \beta_{i,k} \! \gamma_i \! \right) \! = \! \lambda_{k+1} \! \left[\gamma_{0} - \left(\mu + \sum_{i = 1}^{k-1} \beta_{i,k} \gamma_i\right)\right].
	\end{align}
	Letting $\lambda_{k+1} = (1-\alpha_k) \lambda_k$ and considering \eqref{alpha_k_intuition}, we have
	\begin{align}
		\nonumber
		\alpha_k &= 1 - \frac{\lambda_{k+1}}{\lambda_k} = \sqrt{\frac{\gamma_{k+1}}{L}} = \sqrt{\frac{\mu \! + \! \sum_{i = 1}^{k-1} \! \beta_{i,k} \gamma_i}{L} \! + \! \frac{\gamma_{k+1} \! - \! \left( \! \mu \! + \! \sum_{i = 1}^{k-1} \! \beta_{i,k} \gamma_i \! \right)}{L}} \\ \nonumber
		&\stackrel{\eqref{FGM_conv_eq_1}}{=} \! \sqrt{ \! \frac{\mu \! + \! \sum_{i = 1}^{k-1} \! \beta_{i,k} \gamma_i}{L} \! + \! \lambda_{k+1}\frac{\gamma_{0} \! - \! \left( \! \mu \! + \! \sum_{i = 1}^{k-1} \! \beta_{i,k} \gamma_i \! \right)}{L}}. 
	\end{align}
	Moreover,
	\begin{align}
		\nonumber
		\frac{\lambda_k  - \lambda_{k+1}}{\lambda_k} &= \sqrt{\lambda_{k+1}}  \sqrt{\frac{\mu \! + \! \sum_{i = 1}^{k-1} \! \beta_{i,k} \gamma_i}{\lambda_{k+1} L} \! + \! \frac{\gamma_{0} \! - \! \left( \! \mu \! + \! \sum_{i = 1}^{k-1} \! \beta_{i,k} \gamma_i \! \right)}{L}}, \\ \label{fdfd}
		\frac{\lambda_k - \lambda_{k+1}}{\lambda_k \lambda_{k+1}} &= \frac{1}{\sqrt{\lambda_{k+1}}}  \sqrt{\frac{\mu \! + \! \sum_{i = 1}^{k-1} \! \beta_{i,k} \gamma_i}{\lambda_{k+1} L} \! + \! \frac{\gamma_{0} \! - \! \left( \! \mu \! + \! \sum_{i = 1}^{k-1} \! \beta_{i,k} \gamma_i \! \right)}{L}}.
	\end{align}
	Then, by writing the LHS of \eqref{fdfd} as $\frac{\lambda_k - \lambda_{k+1}}{\lambda_k \lambda_{k+1}} = \frac{1}{\lambda_{k+1}} - \frac{1}{\lambda_k}$, and utilizing a difference of squares argument, we obtain
	\begin{align}
		\label{convergence_stupid}
		\left(\frac{1}{\sqrt{\lambda_{k+1}}} \! - \! \frac{1}{\sqrt{\lambda_{k}}} \right) \! \left(\frac{1}{\sqrt{\lambda_{k+1}}} + \frac{1}{\sqrt{\lambda_{k}}}\right) \! = \! \frac{1}{\sqrt{\lambda_{k+1}}} \sqrt{\frac{\mu \! + \! \sum_{i = 1}^{k-1} \! \beta_{i,k} \gamma_i}{\lambda_{k+1} L} \! + \! \frac{\gamma_{0} \! - \! \left(\mu \! + \! \sum_{i = 1}^{k-1} \! \beta_{i,k} \! \gamma_i \right)}{L}}. 
	\end{align}
	In \eqref{convergence_stupid}, we can lower bound the LHS by replacing $\frac{1}{\sqrt{\lambda_{k}}}$ with the larger number $\frac{1}{\sqrt{\lambda_{k+1}}}$. This results in 
	\begin{align}
		\label{FGM_conv_eq_2}
		\frac{2}{\sqrt{\lambda_{k+1}}} \left(\frac{1}{\sqrt{\lambda_{k+1}}} - \frac{1}{\sqrt{\lambda_{k}}}\right)  &\geq \frac{1}{\sqrt{\lambda_{k+1}}} \sqrt{\frac{\mu + \sum_{i = 1}^{k-1} \beta_{i,k} \gamma_i}{\lambda_{k+1} L} + \frac{\gamma_{0} - \left(\mu + \sum_{i = 1}^{k-1} \beta_{i,k} \gamma_i\right)}{L}}.
	\end{align}
	Now, letting
	\begin{align}
		\label{xi_k_def}
		\xi_k \triangleq \sqrt{\frac{L}{\left[\left(\mu + \sum_{i = 1}^{k-1} \beta_{i,k} \gamma_i\right) - \gamma_{0} \right] \lambda_{k}}},
	\end{align}
	we can rewrite \eqref{FGM_conv_eq_2} as
	\begin{align}
		\label{sfgm_conv_useless}
		\frac{2}{\sqrt{\lambda_{k+1}}} - \frac{2}{\sqrt{\lambda_{k}}}  &\geq \sqrt{\frac{\left(\mu + \sum_{i = 1}^{k-1} \beta_{i,k} \gamma_i\right) - \gamma_{0}}{L \left(\mu + \sum_{i = 1}^{k-1} \beta_{i,k} \gamma_i\right)}} \sqrt{\frac{\mu L}{L \lambda_{k+1} \left(\mu + \sum_{i = 1}^{k-1} \beta_{i,k} \gamma_i - \gamma_{0} \right )} - 1}.
	\end{align}
	Then, multiplying both sides of \eqref{sfgm_conv_useless} by $\sqrt{\frac{L}{\mu + \sum_{i = 1}^{k-1} \beta_{i,k} \gamma_i - \gamma_{0}}}$, we obtain
	\begin{align}
		\label{FGM_conv_eq_3}
		\xi_{k+1} - \xi_k &\geq \frac{1}{2}\sqrt{1 + \frac{\left(\mu + \sum_{i = 1}^{k-1} \beta_{i,k} \gamma_i\right) \xi_{k+1}^2}{L}}.
	\end{align}

	At this point, we make use of induction to prove the following bound on $\xi_k$
	\begin{align}
		\label{FGM_conv_eq_4}
		\xi_k \geq \frac{\sqrt{2}}{4 \delta} \sqrt{\frac{L}{\mu}} \left[e^{(k+1) \delta} - e^{(k+1) \delta}\right],
	\end{align}
	where $\delta \triangleq \frac{1}{2} \sqrt{\frac{\mu + \sum_{i = 1}^{k-1} \beta_{i,k} \gamma_i}{L}}$. At step $k = 0$ we have
	\begin{align}
		\label{loose_bound}
		\xi_0 \! \! \! &\stackrel{\eqref{xi_k_def}}{=} \! \! \! \sqrt{\frac{L}{(\mu + \gamma_{-1} - \gamma_{0}) \lambda_{0}}} \! = \! \sqrt{\frac{L}{\mu - \gamma_{0}}} \geq \frac{1}{2} \sqrt{\frac{L}{\mu}} \left[e^{\frac{\sqrt{2}}{2}} - e^{-\frac{\sqrt{2}}{2}}\right] \geq \frac{\sqrt{2}}{4 \delta} \sqrt{\frac{L}{\mu}} \left[e^\delta - e^{-\delta}\right],
	\end{align}
	where the second equality is obtained from the assumptions made in Lemma~\ref{SFGM_lemma_2}, i.e., $\lambda_{0} = 1$ and $\gamma_{k} = 0$, $\forall k<0$. From \eqref{scan_funct_cond}, we must have $\gamma_{0} \geq 0$. Setting $\gamma_0 = 0$ in \eqref{loose_bound} and multiplying it with a number that is smaller than $1$, we obtain the first inequality. The last inequality in \eqref{loose_bound} follows because right-hand side (RHS) is increasing in $\delta$, which by construction is always $\delta < \frac{\sqrt{2}}{2}$.
	
	Next, we assume that \eqref{FGM_conv_eq_4} holds at iteration $k$ and prove the same result for step $k+1$ via contradiction. Letting $\omega(t) = \frac{1}{4 \delta} \sqrt{\frac{L}{\mu}} \left[e^{(t+1) \delta} - e^{-(t+1) \delta}\right]$, which is a convex function \cite[Lemma 2.2.4]{Nesterov_book}, we have
	\begin{align}
		\label{FGM_conv_eq_5}
		\omega(t) \leq \xi_k \stackrel{\eqref{FGM_conv_eq_3}}{\leq} \xi_{k+1} - \frac{1}{2}\sqrt{\frac{\left(\mu + \sum_{i = 1}^{k-1} \beta_{i,k} \gamma_i\right) \xi_{k+1}^2}{L} - 1}.
	\end{align}
	Now, suppose $\xi_{k+1} < \omega(t+1)$. Substituting it into \eqref{FGM_conv_eq_5}, we obtain
	\begin{align}
		\omega(t) &\stackrel{\eqref{FGM_conv_eq_5}}{<} \omega(t+1) - \frac{1}{2}\sqrt{\frac{\left(\mu + \sum_{i = 1}^{k-1} \beta_{i,k} \gamma_i\right) \xi_{k+1}^2}{L} - 1}.
	\end{align}
	Then, applying \eqref{FGM_conv_eq_4} and the definition of $\delta$, yields
	\begin{align}
		\omega(t) \! &\leq \! \omega(t \! + \! 1) \! - \! \frac{1}{2}\sqrt{\! 4 \delta^2 \! \! \left[ \! \frac{\sqrt{2}}{4 \delta} \sqrt{\frac{L}{\mu}} \! \left(e^{(t+2) \delta} \! - \! e^{-(t+2) \delta}\right) \! \right]^2 \! \! \! \! - \! 1} \\ \nonumber
		&= \omega(t+1) - \frac{2}{4} \sqrt{\frac{L}{\mu}} \left[e^{(t+2) \delta} + e^{-(t+2) \delta}\right] \\ \nonumber
		&= \omega(t+1) + \omega(t+1)'\left(t - (t+1)\right) \leq \omega(t),
	\end{align}
	where the last inequality follows from the supporting hyperplane theorem of convex functions. Evidently, this leads to a contradiction with our earlier assumption, which implies that $\xi_{k+1} < \omega(k \! + \! 1), \forall k$. Therefore, \eqref{FGM_conv_eq_4} must hold true. 
	
	Setting $\gamma_{0} = 0$ in \eqref{xi_k_def}, we have
	\begin{align}
		\label{ttttttttttttt}
		\lambda_{k} &= \frac{L}{\left(\mu + \sum_{i = 1}^{k-1} \beta_{i,k} \gamma_i\right) \xi_{k}^2} \stackrel{\eqref{FGM_conv_eq_4}}{\leq} \frac{ \mu (4 \delta)^2}{2 \left(\mu + \sum_{i = 1}^{k-1} \beta_{i,k} \gamma_i\right) \left[e^{(k+1) \delta} - e^{(k+1) \delta}\right]^2},
	\end{align}
	Lastly, applying the definition of $\delta$ in \eqref{ttttttttttttt}, we obtain the first inequality in  \eqref{FGM_conv_eq_66}.
	
	Now, we focus on obtaining the second inequality in \eqref{FGM_conv_eq_66}. We start by abbreviating $\mathbb{A} = \frac{\mu + \sum_{i = 1}^{k-1} \beta_{i,k} \gamma_i}{L}$, and consider the following
	\begin{align}
		\label{non-strongly-cvx-lambda}
		\left(e^{\frac{k + 1}{2} \sqrt{\mathbb{A}}} - e^{-\frac{k + 1}{2} \sqrt{\mathbb{A}}}\right)^2 &= e^{\left(k+1\right) \sqrt{\mathbb{A}}} - e^{-\left(k+1\right) \sqrt{\mathbb{A}}} - 2 \stackrel{}{=} 2*\text{cosh}\left(\sqrt{\mathbb{A}} \left(k+1\right) - 2\right).
	\end{align}
	Utilizing the Taylor expansion of the hyperbolic cosine function, we obtain
	\begin{align}
		\label{to be truncated}
		\left(e^{\frac{k + 1}{2} \sqrt{\mathbb{A}}} - e^{-\frac{k + 1}{2} \sqrt{\mathbb{A}}}\right)^2 &= -2 + 2 + 2 \frac{\mathbb{A} \left(k+1\right)^2}{2} + 2  \frac{\mathbb{A}^2 \left(k+1\right)^4}{4!} + \ldots \text{.}
	\end{align}
	Substituting the abbreviation made for $\mathbb{A}$ and truncating the RHS of \eqref{to be truncated}, we obtain
	\begin{align}
		\label{to be substituted}
		\left(e^{\frac{k + 1}{2} \sqrt{\frac{\mu + \sum_{i = 1}^{k-1} \beta_{i,k} \gamma_i}{L}}} - e^{-\frac{k + 1}{2} \sqrt{\frac{\mu + \sum_{i = 1}^{k-1} \beta_{i,k} \gamma_i}{L}}}\right)^2 \geq \frac{\mu + \sum_{i = 1}^{k-1} \beta_{i,k} \gamma_i}{L} \left(k+1\right)^2
	\end{align}
	Then, substituting the lower bound \eqref{to be substituted} into the denominator of the first inequality of \eqref{FGM_conv_eq_66}, we obtain the desired result.  
	
	%%%%%%%%%%%%%%%%%%%%%%%%%%%%%%%%%%%%%%%%%%%%%%%%%%%%%%%%%%%%
	\section{Proof of Theorem 2}
	\label{Proof of Theorem 2}
	%%%%%%%%%%%%%%%%%%%%%%%%%%%%%%%%%%%%%%%%%%%%%%%%%%%%%%%%%%%%
	Combining the result of Theorem~1 and the inequality $f(x_0) - f^* \leq \frac{L}{2} ||x_0 - x^*||^2$, we obtain
	\begin{align}
		\label{relax}
		f(x_k) \! - \! f(x^*) \! \leq \! \frac{\lambda_k L}{2} ||x_0 - x^*||^2 - (1-\lambda_k)\psi_k(x^*)
	\end{align} 
	Substituting the bound on the term $\lambda_k$ obtained in \eqref{FGM_conv_eq_66} in the paper, yields \eqref{tttr}, in the paper. Then, relaxing the upper bound in \eqref{relax}, yields
	\begin{align}
		\label{FGM_conv_eq_7}
		f(x_k) - f(x^*) \leq \frac{2 \mu ||x_0 - x^*||^2}{e^{(k+1) \sqrt{\frac{\mu + \sum_{i = 1}^{k-1} \beta_{i,k} \gamma_{i}}{L}}} - 1}.
	\end{align}
	
	Therefore, in view of \eqref{eps}, our problem will be solved for 
	\begin{align}
		k_{SFGM} > \sqrt{\frac{L}{\mu + \sum_{i = 1}^{k-1} \beta_{i,k} \gamma_{i}}} \text{ln} \left(1 + \frac{2 \mu R_0^2}{\epsilon}\right).
	\end{align}
	Moreover, we have 
	\begin{align}
		\text{ln} \! \left( \! 1 + \frac{2 \mu R_0^2}{\epsilon} \!\right) \! \stackrel{\eqref{eps}}{\leq} \! \text{ln} \left(\frac{\mu R^2}{2 \epsilon} + \frac{2 \mu R_0^2}{\epsilon}\right) \! = \! \text{ln} \left(\frac{5 \mu R_0^2}{2 \epsilon}\right).
	\end{align} 
	Finally, the lower bound on the number of iterations for Algorithm~1 is
	\begin{align}
		\label{sfgm_bound}
		k_{SFGM} &\geq \sqrt{\frac{L}{\mu + \sum_{i = 1}^{k-1} \beta_{i,k} \gamma_{i}}} \left(\text{ln} \left(\frac{\mu R_0^2}{2 \epsilon}\right) + \text{ln}(5)\right) \\ \label{lower_bound__} &\rightarrow \sqrt{\frac{L}{2 \mu}} \left(\text{ln} \left(\frac{\mu R_0^2}{2 \epsilon}\right) + \text{ln}(5)\right).
	\end{align}
	In the paper, we also present a scheme that converges quickly to the lower bound \eqref{lower_bound__}.
	
	From the lower complexity bounds for the class of smooth and strongly convex functions, we have that
	\begin{align}
		\label{k_bound}
		k_{bound} &\geq \frac{\sqrt{L/\mu} - 1}{4} \text{ln} \left(\frac{\mu R_0^2}{2 \epsilon}\right).
	\end{align}
	Clearly, the bound obtained in \eqref{sfgm_bound} is proportional to \eqref{k_bound}. Therefore, we can conclude that our proposed method is optimal.

	\bibliographystyle{IEEEtran}
	%\bibliographystyle{plain}
	%\bibliography{IEEEabrv,ref}
	\bibliography{IEEEabrv,Heathabrv,refv9,Papers}
\end{document}